\documentclass[reqno]{amsart}
\usepackage{amsmath, amsthm, amssymb, amstext}

\usepackage[left=3cm,right=3cm,top=2cm,bottom=2cm,includeheadfoot]{geometry}
\usepackage{hyperref,xcolor}
\hypersetup{
	pdfborder={0 0 0},
	colorlinks,
}
\usepackage{enumitem}
\newtheorem{theorem}{Theorem}
\newtheorem{remark}[theorem]{Remark}
\newtheorem{lemma}[theorem]{Lemma}
\newtheorem{proposition}[theorem]{Proposition}

\newtheorem{definition}[theorem]{Definition}

\DeclareMathOperator*{\esssup}{ess\,sup}

\newcommand*\diff{\mathrm{d}}
\newcommand{\round}[1]{\left(#1\right)}

\DeclareMathOperator*{\divergenz}{div}              %

\newcommand{\Lp}[1]{L^{#1}(\Omega)}

\newcommand{\Lprand}[1]{L^{#1}(\Gamma)}
\newcommand{\Wp}[1]{W^{1,#1}(\Omega)}

\newcommand{\Wpzero}[1]{W^{1,#1}_0(\Omega)}

\newcommand{\eps}{\varepsilon}
\newcommand{\ph}{\varphi}

\newcommand{\into}{\int_{\Omega}}

\newcommand{\Linf}{L^{\infty}(\Omega)}
\newcommand{\close}{\overline{\Omega}}

\newcommand{\cprime}{$'$}

\renewcommand{\l}{\left}
\renewcommand{\r}{\right}

\numberwithin{theorem}{section}
\numberwithin{equation}{section}

\newcommand{\R}{{\mathbb R}}

\newcommand{\N}{{\mathbb N}}

\newcommand{\Assg}[1]{\textup{(g)}}

\title[New embedding results for double phase problems with variable exponents]
{New embedding results for double phase problems with variable exponents and a priori bounds for corresponding generalized double phase problems}

\author[K. Ho]{Ky Ho}
\address[K. Ho]{Institute of Applied Mathematics, University of Economics Ho Chi Minh City, 59C, Nguyen Dinh Chieu Street, District 3, Ho Chi Minh City, Viet Nam}
\email{kyhn@ueh.edu.vn}

\author[P. Winkert]{Patrick Winkert}
\address[P. Winkert]{Technische Universit\"{a}t Berlin, Institut f\"{u}r Mathematik, Stra\ss{}e des 17.\,Juni 136, 10623 Berlin, Germany}
\email{winkert@math.tu-berlin.de}


\begin{document}

\begin{abstract}
	In this paper we present new embedding results for Musielak-Orlicz Sobolev spaces of double phase type. Based on the continuous embedding of $W^{1,\mathcal{H}}(\Omega)$ into $L^{\mathcal{H}_*}(\Omega)$, where $\mathcal{H}_*$ is the Sobolev conjugate function of $\mathcal{H}$, we present much stronger embeddings as known in the literature. Based on these results, we consider generalized double phase problems involving such new type of growth with Dirichlet and nonlinear boundary condition and prove appropriate boundedness results of corresponding weak solutions based on the De Giorgi iteration along with localization arguments.
\end{abstract}

	\maketitle

\section{Introduction}

Recently, Crespo-Blanco-Gasi\'{n}ski-Harjulehto-Winkert \cite{Crespo-Blanco-Gasinski-Harjulehto-Winkert-2022} studied the so-called double phase operator with variable exponents given by
\begin{align}\label{double-phase-variable}
	\divergenz\left(|\nabla u|^{p(x)-2}\nabla u+\mu(x) |\nabla u|^{q(x)-2}\nabla u\right),
	\quad u\in \Wp{\mathcal{H}}
\end{align}
with $p,q \in C(\close)$ such that $1<p(x)<q(x)<N$ for all $x \in
\close$, $0\leq \mu(\cdot) \in \Lp{1}$ and $\Wp{\mathcal{H}}$ is the corresponding Musielak-Orlicz Sobolev space being a uniformly convex space. Under the assumptions above, it is shown that the operator is continuous, bounded and strictly monotone. Moreover, under some Lipschitz continuity properties on the exponents and the weight function, we have the continuous embedding
\begin{align}\label{embedding-sobolev-conjugate}
	\Wp{\mathcal{H}}\hookrightarrow \Lp{\mathcal{H}_*},
\end{align}
where the function $\mathcal{H}_*$ is called the Sobolev conjugate function of $\mathcal{H}$ given  by
\begin{align*}
	\mathcal{H}(x,t)=t^{p(x)}+\mu(x)t^{q(x)}\quad\text{for }(x,t)\in\Omega \times [0,\infty),
\end{align*}
see Definition \ref{def-conjugate-function} for the precise characterization of $\mathcal{H}_*$. The proof of the embedding \eqref{embedding-sobolev-conjugate} is based on general embedding results of Musielak-Orlicz Sobolev spaces obtained by Fan \cite{Fan-2012} under the additional condition
\begin{align*}
	\frac{q^+}{p^-}<1+\frac{1}{N}
\end{align*}
with $q^+$ and $p^-$ being the maximum and minimum of $q$ and $p$ on $\close$, respectively. It is known that the embedding in \eqref{embedding-sobolev-conjugate} is not sharp, see Adams-Fournier \cite{Adams-Fournier-2003}, Donaldson-Trudinger \cite{Donaldson-Trudinger-1971} or Fan \cite{Fan-2012}. In the case of sharp embedding results from Orlicz Sobolev spaces into Orlicz spaces we refer to the work of Cianchi \cite{Cianchi-1996}. So far, there does not exist any generalization of such sharp embeddings to Musielak-Orlicz Sobolev spaces.

The main objective of the paper is twofold. In the first part we want to discuss how we can obtain better embedding results from $\Wp{\mathcal{H}}$ into $\Lp{\ph}$ by using the embedding \eqref{embedding-sobolev-conjugate}. It will be seen that we get indeed a much better continuous embedding of the form
\begin{align}\label{embedding-introduction}
	\Wp{\mathcal{H}} \hookrightarrow \Lp{\mathcal{G}^*},
\end{align}
with
\begin{align}\label{critical-phi-functions}
	\mathcal{G}^*(x,t):=t^{p^*(x)}+\mu(x)^{\frac{q^*(x)}{q(x)}}t^{q^*(x)}, \quad (x,t)\in \overline{\Omega}\times [0,\infty),
\end{align}
where for a function $r\in C(\close)$ with $1<r(x)<N$ for all $x\in\close$, the critical exponent $r^*(\cdot)$ is  given by
\begin{align*}
	r^*(x):=\frac{Nr(x)}{N-r(x)}, \quad x\in\close.
\end{align*}
In addition we are able to prove that the exponent $\frac{q^*(x)}{q(x)}$ in \eqref{critical-phi-functions} of $\mu$ is optimal under all possible exponents so that \eqref{embedding-introduction} hold true. However, we do not know if the embedding in \eqref{embedding-introduction} is sharp under all generalized $\Phi$-functions. In the first part we furthermore obtain trace embeddings from $\Wp{\mathcal{H}}$ into $L^\varphi(\partial\Omega)$. Based on a general trace embedding result for Musielak-Orlicz Sobolev spaces obtained by Liu-Wang-Zhao  \cite{Liu-Wang-Zhao-2016}, we show the following critical trace embedding
\begin{equation}\label{embedding-introduction-T}
	\Wp{\mathcal{H}} \hookrightarrow L^{\mathcal{T}^*}(\partial\Omega),
\end{equation}
with
\begin{align}\label{critical-phi-functions-T}
	\mathcal{T}^*(x,t):=t^{p_*(x)}+\mu(x)^{\frac{q_*(x)}{q(x)}}t^{q_*(x)},\quad (x,t)\in \overline{\Omega}\times [0,\infty),
\end{align}
where for a function $r\in C(\close)$ with $1<r(x)<N$ for all $x\in\close$, the critical exponent $r_*(\cdot)$ is given by
\begin{align*}
	r_*(x):=\frac{(N-1)r(x)}{N-r(x)}, \quad x\in\close.
\end{align*}
We also prove corresponding ``subcritical'' embeddings related to \eqref{embedding-introduction} and \eqref{embedding-introduction-T} which turn out to be compact.

In the second part of the paper, based on the new embedding results in \eqref{embedding-introduction} and \eqref{embedding-introduction-T}, we study the boundedness of weak solutions to Dirichlet and Neumann problems of the form
\begin{equation}\label{D}
	\begin{aligned}
		-\operatorname{div}\mathcal{A}(x,u,\nabla u)& =\mathcal{B}(x,u,\nabla u)\quad && \text{in } \Omega,\\
		u & = 0 &&\text{on } \partial \Omega,
	\end{aligned}
\end{equation}
and
\begin{equation}\label{N}
	\begin{aligned}
		-\operatorname{div}\mathcal{A}(x,u,\nabla u)& =\mathcal{B}(x,u,\nabla u)\quad && \text{in } \Omega,\\
		\mathcal{A}(x,u,\nabla u)\cdot \nu & = \mathcal{C}(x,u) &&\text{on } \partial \Omega,
	\end{aligned}
\end{equation}
where $\Omega $ is a bounded domain in $\mathbb{R}^{N}$, $N\geq 2$, with Lipschitz boundary $\Gamma:=\partial\Omega$, $\nu(x)$ denotes the outer unit normal of $\Omega$ at $x\in \Gamma$ and the functions  $\mathcal{A}\colon  \Omega\times \mathbb{R} \times \mathbb R^N \to
\mathbb R^N$, $\mathcal{B}\colon\Omega\times\R\times\R^N\to\R$ and $\mathcal{C}\colon\Gamma\times\R\to\R$ are Carath\'{e}odory functions that fulfill structure conditions as developed in \eqref{critical-phi-functions} and \eqref{critical-phi-functions-T}, see hypotheses \eqref{D1}, \eqref{D2}, \eqref{N1} and \eqref{N2}. Our results are based on the so-called De Giorgi-Nash-Moser theory, which provides iterative methods based on truncation techniques to get a priori bounds for certain equations, see the works of De Giorgi \cite{De-Giorgi-1957}, Nash \cite{Nash-1958} and Moser \cite{Moser-1960}. The techniques developed in these papers provided powerful tools to prove local and global boundedness, the Harnack and the weak Harnack inequality and the H\"older continuity of weak solutions. For more information we refer to the monographs of Gilbarg-Trudinger \cite{Gilbarg-Trudinger-1983}, Lady{\v{z}}enskaja-Ural{\cprime}ceva \cite{Ladyzenskaja-Uralceva-1968},  Lady{\v{z}}enskaja-Solonnikov-Ural{\cprime}ceva \cite{Ladyzenskaja-Solonnikov-Uralceva-1968} and Lieberman \cite{Lieberman-1996}. Our proofs for $L^\infty$-bounds are mainly based on the papers of Ho-Kim \cite{Ho-Kim-2019}, Ho-Kim-Winkert-Zhang \cite{Ho-Kim-Winkert-Zhang-2022} and Winkert-Zacher \cite{Winkert-Zacher-2012,Winkert-Zacher-2015}. We also mention the boundedness results in the works of Barletta-Cianchi-Marino \cite{Barletta-Cianchi-Marino-2022} (for problems in Orlicz spaces), Gasi\'{n}ski-Winkert \cite{Gasinski-Winkert-2020a, Gasinski-Winkert-2021} (for double phase Dirichlet and Neumann problems), Ho-Sim \cite{Ho-Sim-2017} (for weighted problems), Kim-Kim-Oh-Zeng \cite{Kim-Kim-Oh-Zeng-2022} (for variable exponent double phase problems with a growth less than $p^*(\cdot)$), Marino-Winkert \cite{Marino-Winkert-2020, Marino-Winkert-2019} (for critical problems in $\Wp{p}$) and Winkert \cite{Winkert-2010} (for subcritical problems in $\Wp{p}$).

Coming back to the operator \eqref{double-phase-variable}, it is clear that this is a natural extension of the classical double phase operator when $p$ and $q$ are constants, namely
\begin{align}\label{double-phase}
	\divergenz\Big(|\nabla u|^{p-2}\nabla u+\mu(x) |\nabla u|^{q-2}\nabla u\Big).
\end{align}
It is easy to verify, that if  $\inf_{\close} \mu\geq \mu_0>0$ or $\mu\equiv 0$, then the operator in \eqref{double-phase-variable} becomes the weighted $(q(\cdot),p(\cdot))$-Laplacian or the $p(\cdot)$-Laplacian, respectively. The energy functional $I\colon\Wp{\mathcal{H}}\to \R$ related to the double phase
operator \eqref{double-phase-variable} is given by
\begin{align}\label{integral_minimizer}
	I(u)=\into \l( \frac{|\nabla u|^{p(x)}}{p(x)}
	+ \mu(x) \frac{|\nabla u|^{q(x)}}{q(x)}\r)\diff x,
\end{align}
where the integrand
\begin{align*}
	\mathcal{R}(x,\xi)=\frac{1}{p(x)}|\xi|^{p(x)}+\frac{\mu(x)}{q(x)}|\xi|^{q(x)}
	\quad \text{for all }(x,\xi) \in \Omega\times \R^N
\end{align*}
of $I$ has unbalanced growth if $0\leq \mu(\cdot) \in \Lp{\infty}$, that is,
\begin{align*}
	c_1|\xi|^{p(x)} \leq \mathcal{R}(x,\xi) \leq c_2 \l(1+|\xi|^{q(x)}\r)
\end{align*}
for a.\,a.\,$x\in\Omega$ and for all $\xi\in\R^N$ with $c_1,c_2>0$. The main feature of the functional $I$ given in \eqref{integral_minimizer} is the change of ellipticity on the set where the weight function is zero, that is, on the set $\{x\in \Omega\,:\, \mu(x)=0\}$. This means, that the energy density of $I$ exhibits ellipticity in the gradient of order $q(x)$ in the set $\{x\in\Omega\,:\,\mu(x)>\eps\}$ for any fixed $\eps>0$ and of order $p(x)$ on the points $x$ where $\mu(x)$ vanishes. Therefore, the integrand $\mathcal{R}$ switches between two different phases of elliptic behaviours and so it is called double phase.

We point out that Zhikov \cite{Zhikov-1986} was the first who considered functionals defined by
\begin{align*}
	u \mapsto \int \l( |\nabla u|^{p}
	+ \mu(x) |\nabla u|^{q}\r)\diff x,
\end{align*}
whose integrands change their ellipticity according to a point in order to provide
models for strongly anisotropic materials. This type of functional given in \eqref{integral_minimizer} has been treated in many papers concerning regularity of local minimizers. In this direction we mention the works of Baroni-Colombo-Mingione \cite{Baroni-Colombo-Mingione-2015,Baroni-Colombo-Mingione-2016,Baroni-Colombo-Mingione-2018}, Baroni-Kuusi-Mingione \cite{Baroni-Kuusi-Mingione-2015}, Byun-Oh \cite{Byun-Oh-2020}, Colombo-Mingione \cite{Colombo-Mingione-2015a,Colombo-Mingione-2015b}, De Filippis \cite{De-Filippis-2018}, De Filippis-Palatucci \cite{De-Filippis-Palatucci-2019}, Harjulehto-H\"{a}st\"{o}-Toivanen \cite{Harjulehto-Hasto-Toivanen-2017},
Marcellini \cite{Marcellini-1991,Marcellini-1989b}, Ok \cite{Ok-2018,Ok-2020}, Ragusa-Tachikawa \cite{Ragusa-Tachikawa-2016,Ragusa-Tachikawa-2020}
and the references therein. Moreover, recent results for nonuniformly elliptic variational problems and nonautonomous functionals can be found in the papers of Beck-Mingione \cite{Beck-Mingione-2020,Beck-Mingione-2019},
De Filippis-Mingione \cite{De-Filippis-Mingione-2020,DeFilippis-Mingione-ARMA-2021,DeFilippis-Mingione-2021,DeFilippis-Mingione-JGA-2020,DeFilippis-Mingione-JGA-2020-2} and H\"{a}st\"{o}-Ok \cite{Hasto-Ok-2019}. For other applications in physics and engineering of double phase differential operators and related energy functionals given in \eqref{double-phase} and \eqref{integral_minimizer}, respectively,  we refer to the works of Bahrouni-R\u{a}dulescu-Repov\v{s} \cite{Bahrouni-Radulescu-Repovs-2019} on transonic flows, Benci-D'Avenia-Fortunato-Pisani \cite{Benci-DAvenia-Fortunato-Pisani-2000} on quantum physics and  Cherfils-Il\cprime yasov \cite{Cherfils-Ilyasov-2005} on reaction diffusion systems. For example, in the elasticity theory, the modulating coefficient $\mu(\cdot)$ dictates the geometry of composites made of two different materials with distinct power hardening exponents $q(\cdot)$ and $p(\cdot)$,
see Zhikov \cite{Zhikov-2011}.

At the end we also want to mention some recent existence results in the direction of double phase problems developed with different methods and techniques. We refer to the papers of Bahrouni-R\u{a}dulescu-Winkert \cite{Bahrouni-Radulescu-Winkert-2020} (Baouendi-Grushin operator),
Colasuonno-Squassina \cite{Colasuonno-Squassina-2016} (double phase eigenvalue problems), Farkas-Winkert \cite{Farkas-Winkert-2021} (Finsler double phase problems), Gasi\'{n}ski-Papageorgiou \cite{Gasinski-Papageorgiou-2019} (locally Lipschitz right-hand sides), Gasi\'{n}ski-Winkert \cite{Gasinski-Winkert-2020b,Gasinski-Winkert-2021} (convection and superlinear problems), Liu-Dai \cite{Liu-Dai-2018} (Nehari manifold treatment), Papageorgiou-R\u{a}dulescu-Repov\v{s} \cite{Papageorgiou-Radulescu-Repovs-2019-a, Papageorgiou-Radulescu-Repovs-2020} (property of the spectrum and ground state solutions), Perera-Squassina \cite{Perera-Squassina-2018} (Morse theory for double phase problems), Zhang-R\u{a}dulescu \cite{Zhang-Radulescu-2018} and Shi-R\u{a}dulescu-Repov\v{s}-Zhang \cite{Shi-Radulescu-Repovs-Zhang-2020} (double phase anisotropic variational problems with variable exponents), Zeng-Bai-Gasi\'{n}ski-Winkert \cite{Zeng-Bai-Gasinski-Winkert-2020} (implicit obstacle double phase problems), Zeng-R\u{a}dulescu-Winkert \cite{Zeng-Radulescu-Winkert-2022} (implicit obstacle double phase problems with mixed boundary condition), see also the references therein. It is worth pointing out that while these works treat  double phase problems in terms of two exponents $p(\cdot)$ and $q(\cdot)$ with $p(\cdot)<q(\cdot)$, its nonlinear terms have a growth that does not exceed $p^*(\cdot)$. Our new embeddings will provide a necessary ingredient to study double phase problems which have a growth between $p^*(\cdot)$ and $q^*(\cdot)$.

The paper is organized as follows. In Section \ref{Section-2} we recall some properties of the double phase operator with variable exponents and present relevant embedding results. Section \ref{Section-3} is devoted to the study of the critical and subcritical embeddings mentioned above, see Propositions \ref{prop_C-do-E}, \ref{prop_C-do-E-Op} and \ref{prop_C-T-E} for the critical case and Propositions \ref{prop_S-C-E} and \ref{prop_S-C-T-E} for the subcritical case. In Section \ref{Section-4} we consider problems \eqref{D} and \eqref{N} where we suppose subcritical growth and prove a priori bounds for corresponding weak solutions, see Theorems \ref{D.a-priori} and \ref{N.a-priori}. Finally, using the critical embedding in \eqref{embedding-introduction}, we also develop boundedness results for weak solutions of \eqref{D} and \eqref{N} in this case, see Theorems \ref{Theo.D} and \ref{Theo.N} in Section \ref{Section-5}.

\section{Preliminaries and Notations}\label{Section-2}

In this section we recall the main properties to Musielak-Orlicz Sobolev spaces and the double phase operator with variable exponents. These results are mainly taken from Crespo-Blanco-Gasi\'{n}ski-Harjulehto-Winkert \cite{Crespo-Blanco-Gasinski-Harjulehto-Winkert-2022}, we refer also to the books of Diening-Harjulehto-H\"{a}st\"{o}-R$\mathring{\text{u}}$\v{z}i\v{c}ka \cite{Diening-Harjulehto-Hasto-Ruzicka-2011}, Harjulehto-H\"{a}st\"{o} \cite{Harjulehto-Hasto-2019}, Musielak \cite{Musielak-1983}, Papageorgiou-R\u{a}dulescu-Repov\v{s} \cite{Papageorgiou-Radulescu-Repovs-2019-b}, R\u{a}dulescu-Repov\v{s} \cite{Radulescu-Repovs-2015} and the papers of Colasuonno-Squassina \cite{Colasuonno-Squassina-2016}, Fan \cite{Fan-2012}, Fan-Zhao \cite{Fan-Zhao-2001}, Kov{\'a}{\v{c}}ik-R{\'a}kosn{\'{\i}}k \cite{Kovacik-Rakosnik-1991} and Liu-Wang-Zhao \cite{Liu-Wang-Zhao-2016}.

Let $\Omega$ be a bounded domain in $\mathbb{R}^N$ with Lipschitz boundary $\Gamma:=\partial\Omega$ and let $M(\Omega)$ be the space of all measurable functions $u\colon \Omega\to\R$.

We start with the following definition.
\begin{definition}
	$~$
	\begin{enumerate}
		\item[\textnormal{(i)}]
		A continuous and convex function $\ph\colon[0,\infty)\to[0,\infty)$ is said to be a $\Phi$-function if $\ph(0)=0$ and $\ph(t)>0$ for all $t >0$.
		\item[\textnormal{(ii)}]
		A function $\ph\colon\Omega \times [0,\infty)\to[0,\infty)$ is said to be a generalized $\Phi$-function if $\ph(\cdot,t)\in M(\Omega)$ for all $t\geq 0$ and $\ph(x,\cdot)$ is a $\Phi$-function for a.\,a.\,$x\in\Omega$. We denote the set of all generalized $\Phi$-functions on $\Omega$ by $\Phi(\Omega)$.
		\item[\textnormal{(iii)}]
		A function $\ph\in\Phi(\Omega)$ is locally integrable if $\ph(\cdot,t) \in L^{1}(\Omega)$ for all $t>0$.
		\item[\textnormal{(iv)}]
		Given $\ph, \psi \in \Phi(\Omega)$, we say that $\ph$ is weaker than $\psi$, denoted by $\ph \prec \psi$, if there exist two positive constants $C_1, C_2$ and a nonnegative function $h\in\Lp{1}$ such that
		\begin{align*}
			\ph(x,t) \leq C_1 \psi(x,C_2t)+h(x)
		\end{align*}
		for a.\,a.\,$x\in \Omega$ and for all $t \in[0,\infty)$.
		\item[\textnormal{(v)}]
			Let $\phi, \psi \in \Phi(\Omega)$. We say that $\phi$ increases essentially slower than $\psi$ near infinity, if for any $k>0$
			\begin{align*}
				\lim_{t\to\infty} \frac{\phi(x,kt)}{\psi(x,t)}=0 \quad \text{uniformly for a.\,a.\,}x\in\Omega.
			\end{align*}
			We write $\phi \ll \psi$.
	\end{enumerate}
\end{definition}
For a given $\ph \in \Phi(\Omega)$ we define the corresponding modular $\rho_\ph$ by
\begin{align*}
	\rho_\ph(u):= \into \ph\left(x,|u|\right)\,\diff x.
\end{align*}
Then, the Musielak-Orlicz space $L^\ph(\Omega)$ is defined by
\begin{align*}
	L^\ph(\Omega):=\left \{u \in M(\Omega)\,:\, \text{there exists }\alpha>0 \text{ such that }\rho_\ph(\alpha u)< +\infty \right \}
\end{align*}
equipped with the norm
\begin{align*}
	\|u\|_{\ph,\Omega}:=\inf \left\{\alpha >0 \, : \, \rho_\ph \left(\frac{u}{\alpha}\right)\leq 1\right\}.
\end{align*}
Similarly, we define Musielak-Orlicz spaces $L^{\ph}(\Gamma)$ on the boundary equipped with the norm $\|\cdot\|_{\ph,\Gamma}$, where we use the $(N-1)$-dimensional Hausdorff surface measure $\sigma$ on $\R^N$.

The following proposition can be found in Musielak \cite[Theorem 7.7 and Theorem 8.5]{Musielak-1983}.

\begin{proposition}\label{prop_complete}$~$
	\begin{enumerate}
		\item[\textnormal{(i)}]
			Let $\ph \in \Phi(\Omega)$. Then the Musielak-Orlicz space $\Lp{\ph}$ is complete with respect to the norm $\|\cdot\|_{\varphi,\Omega}$, that is, $\left(\Lp{\ph},\|\cdot\|_{\varphi,\Omega}\right)$ is a Banach space.
		\item[\textnormal{(ii)}]
			Let $\ph,\psi \in \Phi(\Omega)$ be locally integrable with $\ph \prec \psi$. Then
			\begin{align*}
				\Lp{\psi} \hookrightarrow \Lp{\ph}.
			\end{align*}
	\end{enumerate}
\end{proposition}

Next, we define Musielak-Orlicz Sobolev spaces. For this purpose, we need the following definition.
\begin{definition}$~$
	\begin{enumerate}
		\item[\textnormal{(i)}]
			The function $\ph\colon [0,\infty) \to [0,\infty)$ is called $N$-function if it is a $\Phi$-function such that
			\begin{align*}
				\lim_{t\to 0^+} \frac{\ph(t)}{t}=0
				\quad\text{and}\quad
				\lim_{t\to\infty} \frac{\ph(t)}{t}=\infty.
			\end{align*}
		\item[\textnormal{(ii)}]
			We call a function $\ph\colon\Omega\times \R\to[0,\infty)$ a generalized $N$-function if $\ph(\cdot,t)$ is measurable for all $t \in \R$ and $\ph(x,\cdot)$ is a $N$-function for a.\,a.\,$x\in\Omega$. We denote the class of all generalized $N$-functions by $N(\Omega)$.
	\end{enumerate}
\end{definition}
Let $\ph \in N(\Omega)$ be locally integrable. The Musielak-Orlicz Sobolev space $\Wp{\ph}$ is defined by
\begin{align*}
	\Wp{\ph} := \left \{u \in \Lp{\ph} \,:\, |\nabla u| \in \Lp{\ph} \right\}
\end{align*}
equipped with the norm
\begin{align*}
	\|u\|_{1,\ph} = \|u\|_\ph+\|\nabla u\|_\ph,
\end{align*}
where $\|\nabla u\|_\ph=\| \, |\nabla u| \,\|_\ph$.  The completion of $C^\infty_0(\Omega)$ in $\Wp{\ph}$ is denoted by $\Wpzero{\ph}$.

The next theorem can be found in Musielak \cite[Theorem 10.2]{Musielak-1983} and  Fan \cite[Propositions 1.7 and 1.8]{Fan-2012}.

\begin{theorem}
	Let $\ph\in N(\Omega)$ be locally integrable such that
	\begin{align*}
		\inf_{x\in\Omega} \ph(x,1)>0.
	\end{align*}
	Then the spaces $\Wp{\ph}$ and $\Wpzero{\ph}$ are separable Banach spaces which are reflexive if $\Lp{\ph}$ is reflexive.
\end{theorem}

Let us now come to our special Musielak-Orlicz Sobolev space and its properties, which was introduced in \cite{Crespo-Blanco-Gasinski-Harjulehto-Winkert-2022}. In the following, for $h \in C(\close)$ we denote
\begin{align*}
	h^{-}:=\inf_{x\in \close} h(x)
	\quad \text{and} \quad
	h^{+}:=\sup_{x\in \close} h(x).
\end{align*}

We suppose the following assumptions:
\begin{enumerate}[label=\textnormal{(H$1$)},ref=\textnormal{H$1$}]
	\item\label{H1}
		$p,q\in C(\close)$ such that $1<p(x)<N$ and $p(x) < q(x)$ for all $x\in\close$ and $0 \leq \mu(\cdot) \in \Lp{1}$.
\end{enumerate}

Under hypothesis \eqref{H1}, let $\mathcal{H} \colon \Omega \times [0,\infty) \to [0,\infty)$ be the nonlinear function defined by
\begin{align}\label{def-H}
	\mathcal{H}(x,t):=t^{p(x)} +\mu(x)t^{q(x)} \quad\text{for all } (x,t)\in \Omega \times [0,\infty).
\end{align}

Recall that the corresponding modular to $\mathcal{H}$ is given by
\begin{align}\label{modular-Lp}
	\rho_{\mathcal{H}}(u) = \into \mathcal{H} (x,|u|)\,\diff x.
\end{align}
Then, the corresponding Musielak-Orlicz space $\Lp{\mathcal{H}}$ is given
by
\begin{align*}
	L^{\mathcal{H}}(\Omega)=\left \{u \in M(\Omega) \,:\,\rho_{\mathcal{H}}(u) < +\infty \right \},
\end{align*}
endowed with the norm
\begin{align*}
	\|u\|_{\mathcal{H}} = \inf \left \{ \tau >0 : \rho_{\mathcal{H}}\left(\frac{u}{\tau}\right) \leq 1  \right \}.
\end{align*}

Next, we can introduce the Musielak-Orlicz Sobolev space $\Wp{\mathcal{H}}$ defined by
\begin{align*}
	\Wp{\mathcal{H}}
	=\left \{u \in L^{\mathcal{H}}(\Omega) \,:\,|\nabla u| \in L^{\mathcal{H}}(\Omega) \right \}
\end{align*}
equipped with the norm
\begin{align*}
	\|u\|_{1,\mathcal{H}} = \|u\|_{\mathcal{H}}+\|\nabla u\|_{\mathcal{H}},
\end{align*}
where $\|\nabla u\|_{\mathcal{H}}=\| \, |\nabla u| \,\|_{\mathcal{H}}$. Moreover, $\Wpzero{\mathcal{H}}$ is the completion of $C^\infty_0(\Omega)$ in $\Wp{\mathcal{H}}$. We know that  $\Lp{\mathcal{H}}$, $\Wp{\mathcal{H}}$ and $\Wpzero{\mathcal{H}}$ are reflexive Banach spaces, see \cite[Proposition 2.12]{Crespo-Blanco-Gasinski-Harjulehto-Winkert-2022}.

The following proposition gives the relation between the modular $\rho_{\mathcal{H}}$ and its norm $\|\cdot\|_\mathcal{H}$, see \cite[Proposition 2.13]{Crespo-Blanco-Gasinski-Harjulehto-Winkert-2022}.

\begin{proposition}\label{prop_mod-nor}
	Let hypotheses \eqref{H1} be satisfied, $u\in\Lp{\mathcal{H}}$ and let $\rho_{\mathcal{H}}$ be defined as in \eqref{modular-Lp}.
	\begin{enumerate}
		\item[\textnormal{(i)}]
		If $u\neq 0$, then $\|u\|_{\mathcal{H}}=\lambda$ if and only if $ \rho_{\mathcal{H}}(\frac{u}{\lambda})=1$.
		\item[\textnormal{(ii)}]
		$\|u\|_{\mathcal{H}}<1$ (resp.\,$>1$, $=1$) if and only if $ \rho_{\mathcal{H}}(u)<1$ (resp.\,$>1$, $=1$).
		\item[\textnormal{(iii)}]
		If $\|u\|_{\mathcal{H}}<1$, then $\|u\|_{\mathcal{H}}^{q^+}\leqslant \rho_{\mathcal{H}}(u)\leqslant\|u\|_{\mathcal{H}}^{p^-}$.
		\item[\textnormal{(iv)}]
		If $\|u\|_{\mathcal{H}}>1$, then $\|u\|_{\mathcal{H}}^{p^-}\leqslant \rho_{\mathcal{H}}(u)\leqslant\|u\|_{\mathcal{H}}^{q^+}$.
			\end{enumerate}
\end{proposition}

On $\Wp{\mathcal{H}}$, we will also work with the following equivalent norm
\begin{align}\label{equivalent-norm-W1H}
	\|u\|_{1,\mathcal{H},\Omega}:=\inf&\left\{\lambda >0 \,:\, \hat{\rho}_{1,\mathcal{H},\Omega}\left(\frac{u}{\lambda}\right)\le1\right\},
\end{align}
where the modular $\hat{\rho}_{1,\mathcal{H},\Omega}$ is given by
\begin{align}\label{modular-W1p}
	\hat{\rho}_{1,\mathcal{H},\Omega}(u) =\into \left[|\nabla u|^{p(x)}+\mu(x)|\nabla u|^{q(x)}+|u|^{p(x)}+\mu(x)|u|^{q(x)}\right]\,\diff x
\end{align}
for $u \in \Wp{\mathcal{H}}$.

The following results can be found in \cite[Proposition 2.14]{Crespo-Blanco-Gasinski-Harjulehto-Winkert-2022}.

\begin{proposition}\label{prop_mod-nor2}
	Let hypotheses \eqref{H1} be satisfied, let $y\in\Wp{\mathcal{H}}$ and let $\hat{\rho}_{1,\mathcal{H},\Omega}$ be defined as in  \eqref{modular-W1p}.
	\begin{enumerate}
		\item[\textnormal{(i)}]
		If $y\neq 0$, then $\|y\|_{1,\mathcal{H},\Omega}=\lambda$ if and only if $ \hat{\rho}_{1,\mathcal{H},\Omega}(\frac{y}{\lambda})=1$.
		\item[\textnormal{(ii)}]
		$\|y\|_{1,\mathcal{H},\Omega}<1$ (resp.\,$>1$, $=1$) if and only if $ \hat{\rho}_{1,\mathcal{H},\Omega}(y)<1$ (resp.\,$>1$, $=1$).
		\item[\textnormal{(iii)}]
		If $\|y\|_{1,\mathcal{H},\Omega}<1$, then $\|y\|_{1,\mathcal{H},\Omega}^{q^+}\leqslant \hat{\rho}_{1,\mathcal{H},\Omega}(y)\leqslant\|y\|_{1,\mathcal{H},\Omega}^{p^-}$.
		\item[\textnormal{(iv)}]
		If $\|y\|_{1,\mathcal{H},\Omega}>1$, then $\|y\|_{1,\mathcal{H},\Omega}^{p^-}\leqslant \hat{\rho}_{1,\mathcal{H},\Omega}(y)\leqslant\|y\|_{1,\mathcal{H},\Omega}^{q^+}$.
	\end{enumerate}
\end{proposition}

When $\mu(\cdot)\equiv 0$, we write $L^{p(\cdot)}(\Omega)$, $L^{p(\cdot)}(\Gamma)$, $W^{1,p(\cdot)}(\Omega)$, $W_0^{1,p(\cdot)}(\Omega)$, $\|\cdot\|_{p(\cdot),\Omega}$, $\|\cdot\|_{p(\cdot),\Gamma}$ and $\|\cdot\|_{1,p(\cdot),\Omega}$ in place of $\Lp{\mathcal{H}}$, $L^{\mathcal{H}}(\Gamma)$, $\Wp{\mathcal{H}}$, $\Wpzero{\mathcal{H}}$, $\|\cdot\|_{\mathcal{H},\Omega}$, $\|\cdot\|_{\mathcal{H},\Gamma}$ and $\|\cdot\|_{1,\mathcal{H},\Omega}$, respectively. For a function $r\in C(\close)$ with $1<r(x)<N$ for all $x\in\close$ we define
\begin{align}\label{critical-eponents}
	r^*(x)=\frac{Nr(x)}{N-r(x)}
	\quad\text{and}\quad
	r_*(x)=\frac{(N-1)r(x)}{N-r(x)}, \quad x\in\close.
\end{align}

The following embedding results can be found in \cite{Crespo-Blanco-Gasinski-Harjulehto-Winkert-2022}.

\begin{proposition}\label{proposition_embeddings}
	Let hypotheses \eqref{H1} be satisfied . Then the following embeddings hold:
	\begin{enumerate}
		\item[\textnormal{(i)}]
		$\Lp{\mathcal{H}} \hookrightarrow \Lp{r(\cdot)}$, $\Wp{\mathcal{H}}\hookrightarrow \Wp{r(\cdot)}$, $\Wpzero{\mathcal{H}}\hookrightarrow \Wpzero{r(\cdot)}$
		are continuous for all $r\in C(\close)$ with $1\leq r(x)\leq p(x)$ for all $x \in \close$;
		\item[\textnormal{(ii)}]
		if $p \in C^{0, 1}(\close)$, then $\Wp{\mathcal{H}} \hookrightarrow \Lp{r(\cdot)}$ and $\Wpzero{\mathcal{H}} \hookrightarrow \Lp{r(\cdot)}$ are continuous for $r \in C(\close)$ with $ 1 \leq r(x) \leq p^*(x)$ for all $x\in \close$;
		\item[\textnormal{(iii)}]
		$\Wp{\mathcal{H}} \hookrightarrow \Lp{r(\cdot)}$ and $\Wpzero{\mathcal{H}} \hookrightarrow \Lp{r(\cdot)}$ are compact for $r \in C(\close) $ with $ 1 \leq r(x) < p^*(x)$ for all $x\in \close$;
		\item[\textnormal{(iv)}]
		if $p \in W^{1,\gamma}(\Omega)$ for some $\gamma>N$, then $\Wp{\mathcal{H}} \hookrightarrow \Lprand{r(\cdot)}$ and $\Wpzero{\mathcal{H}} \hookrightarrow \Lprand{r(\cdot)}$ are continuous for $r \in C(\close)$ with $ 1 \leq r(x) \leq p_*(x)$ for all $x\in \close$;
		\item[\textnormal{(v)}]
		$\Wp{\mathcal{H}} \hookrightarrow \Lprand{r(\cdot)}$ and $\Wpzero{\mathcal{H}} \hookrightarrow \Lprand{r(\cdot)}$ are compact for $r \in C(\close) $ with $ 1 \leq r(x) < p_*(x)$ for all $x\in \close$;
		\item[\textnormal{(vi)}]
		if $\mu \in \Linf$, then $L^{q(\cdot)}(\Omega) \hookrightarrow \Lp{\mathcal{H}}$ is continuous.
	\end{enumerate}
\end{proposition}
\begin{remark}\label{Rmk}
	Note that for a bounded domain $\Omega\subset \R^N$ and $\gamma>N$ it holds
		$C^{0,1}(\close)\subset \Wp{\gamma}$.
\end{remark}

Let us now suppose stronger conditions as in \eqref{H1}:
\begin{enumerate}[label=\textnormal{(H$2$)},ref=\textnormal{H$2$}]
	\item\label{H2}
	$p,q\in C(\close)$ such that $1<p(x)<N$ and $p(x) < q(x)<p^*(x)$ for all $x\in\close$ and $0 \leq \mu(\cdot) \in \Lp{\infty}$.
\end{enumerate}

Under \eqref{H2} we have the following result, see \cite[Proposition 2.18]{Crespo-Blanco-Gasinski-Harjulehto-Winkert-2022}.

\begin{proposition} \label{prop_poincare}
	Let hypothesis \eqref{H2} be satisfied. Then the following hold:
	\begin{enumerate}
		\item[\textnormal{(i)}]
		$\Wp{\mathcal{H}}\hookrightarrow \Lp{\mathcal{H}}$ is a compact embedding;
		\item[\textnormal{(ii)}]
		there exists a constant $C>0$ independent of $u$ such that
		\begin{align*}
			\|u\|_{\mathcal{H}} \leq C\|\nabla u\|_{\mathcal{H}}\quad\text{for all
			} u \in \Wpzero{\mathcal{H}}.
		\end{align*}
	\end{enumerate}
\end{proposition}


Next we mention the following lemma concerning the geometric convergence of sequences of numbers will be the key to our arguments to obtain the boundedness of solutions via the De Giorgi iteration. The proof of the lemma can be found in the paper of Ho-Sim \cite[Lemma 4.3]{Ho-Sim-2015}
\begin{lemma}\label{leRecur}
	Let $\{Z_n\}, n=0,1,2,\ldots,$ be a sequence of positive numbers, satisfying the recursion inequality
	\begin{align*}
		Z_{n+1} \leq K b^n \left (Z_n^{1+\mu_1}+ Z_n^{1+\mu_2} \right ) , \quad n=0,1,2, \ldots,
	\end{align*}
	for some $b>1$, $K>0$ and $\mu_2\geq \mu_1>0$. If
	\begin{align*}
		Z_0 \leq \min \left(1,(2K)^{-\frac{1}{\mu_1}} b^{-\frac{1}{\mu_1^2}}\right)
	\end{align*}
	or
	\begin{align*}
		Z_0 \leq \min  \left((2K)^{-\frac{1}{\mu_1}} b^{-\frac{1}{\mu_1^2}}, (2K)^{-\frac{1}{\mu_2}}b^{-\frac{1}{\mu_1 \mu_2}-\frac{\mu_2-\mu_1}{\mu_2^2}}\right),
	\end{align*}
	then $Z_n \leq 1$ for some $n \in \N \cup \{0\}$. Moreover,
	\begin{align*}
		Z_n \leq \min \left(1,(2K)^{-\frac{1}{\mu_1}} b^{-\frac{1}{\mu_1^2}} b^{-\frac{n}{\mu_1}}\right), \quad \text{ for all }n \geq n_0,
	\end{align*}
	where $n_0$ is the smallest $n \in \N \cup \{0\}$ satisfying $Z_n \leq 1$. In particular, $Z_n \to 0$ as $n \to \infty$.
\end{lemma}
Furthermore, in the next sections we frequently use Young's inequality of the form
\begin{equation*}
	ab\leq \frac{1}{\eta}\eps a^\eta+\frac{\eta-1}{\eta}\eps^{-\frac{1}{\eta-1}}b^{\frac{\eta}{\eta-1}}\leq \eps a^\eta+\eps^{-\frac{1}{\eta-1}}b^{\frac{\eta}{\eta-1}} \quad \text{ for all } a,b\geq0, \, \eps>0, \,\eta>1.
\end{equation*}

Let us now fix our notation. We write $\N_0:=\N \cup \{0\}$ and for a real number $t>1$ we denote by $t':=\frac{t}{t-1}$ the conjugate number of $t$. For a measurable function $v\colon\Omega\to\R$ we set
\begin{align*}
	v_+ := \max \{v,0\}
	\quad\text{and}\quad
	v_- := \max \{-v,0\}.
\end{align*}
By $|E|$ we denote the $N$-dimensional Lebesgue measure of $E\subset\R^N$ and by $|E|_\sigma$ the $(N-1)$-dimensional surface measure of $E\subset\R^N$. For $1\leq\rho\leq\infty$, the space $L^\rho(E)$ is the usual Lebesgue space with norm $\|\cdot\|_{\rho,E}$.

\section{New embedding results for Musielak-Orlicz Sobolev spaces $\Wp{\mathcal{H}}$}\label{Section-3}

In this section, we want to discuss new embedding results for the space $\Wp{\mathcal{H}}$ into an Musielak-Orlicz space $\Lp{\ph}$ for a suitable $\Phi$-function $\ph$. These results extend those in Proposition \ref{proposition_embeddings}.

First, we are going to introduce the Sobolev conjugate function of $\mathcal{H}$. We define for all $x\in \Omega$
\begin{align*}
	\mathcal{H}_1(x,t)=
	\begin{cases}
		t \mathcal{H}(x,1) &\text{if }0 \leq t \leq 1,\\
		\mathcal{H}(x,t)&\text{if } t > 1.
	\end{cases}
\end{align*}

It is well known, since $\Omega$ is a bounded domain, that $\Lp{\mathcal{H}}=\Lp{\mathcal{H}_1}$ and $\Wp{\mathcal{H}}=\Wp{\mathcal{H}_1}$, see Musielak \cite{Musielak-1983}. Hence, for embedding results of $\Wp{\mathcal{H}}$ we
may use $\mathcal{H}_1$ instead of $\mathcal{H}$. For simplification, we write $\mathcal{H}$ instead of $\mathcal{H}_1$.

We start with the following definition.

\begin{definition}\label{def-conjugate-function}
	We denote by $\mathcal{H}^{-1}(x,\cdot)\colon [0,\infty) \to [0,\infty)$ for all $x \in \overline{\Omega}$ the inverse function of $\mathcal{H}(x,\cdot)$. Furthermore, we define $\mathcal{H}_*^{-1}\colon \overline{\Omega} \times [0,\infty) \to [0,\infty)$ by
	\begin{align}\label{form-H_*^-1}
		\mathcal{H}^{-1}_*(x,s)=\int^s_0 \frac{\mathcal{H}^{-1}(x,\tau)}{\tau^{\frac{N+1}{N}}}\,\diff \tau \quad\text{for all }(x,s) \in \overline{\Omega}\times [0,\infty),
	\end{align}
	where $\mathcal{H}_*\colon (x,t) \in \overline{\Omega}\times[0,\infty) \to s\in [0,\infty)$ is such that $\mathcal{H}^{-1}_*(x,s)=t$. The function $\mathcal{H}_*$ is called the Sobolev conjugate function of $\mathcal{H}$.
\end{definition}

In order to have further properties on $\Wp{\mathcal{H}}$ and $\Wpzero{\mathcal{H}}$, we suppose the following stronger assumptions as those in \eqref{H1} and \eqref{H2}.

\begin{enumerate}[label=\textnormal{(H$3$)},ref=\textnormal{H$3$}]
	\item\label{H3}
	$p,q\in C^{0,1}(\close)$ such that $1<p(x)<q(x)<N$ for all $x\in\close$, $\left(\frac{q}{p}\right)^+<1+\frac{1}{N}$ and $0 \leq \mu(\cdot)\in C^{0,1}(\close)$.
\end{enumerate}

The next proposition, obtained in \cite[Proposition 2.21]{Crespo-Blanco-Gasinski-Harjulehto-Winkert-2022}, provides fundamental embedding results on $\Wp{\mathcal{H}}$ and $\Wpzero{\mathcal{H}}$.

\begin{proposition}\label{prop_dom-emb}
	Let hypotheses \eqref{H3} be satisfied. Then the following hold:
	\begin{enumerate}
		\item[\textnormal{(i)}]
			$\Wp{\mathcal{H}}\hookrightarrow \Lp{\mathcal{H}_*}$ continuously.
		\item[\textnormal{(ii)}]
			Let $\mathcal{K}\colon\Omega \times [0,\infty)\to [0,\infty)$ be continuous such that $\mathcal{K} \in N(\Omega)$ and $\mathcal{K}\ll \mathcal{H}_*$, then $\Wp{\mathcal{H}}\hookrightarrow \Lp{\mathcal{K}}$ compactly.
		\item[\textnormal{(iii})]
			It holds $\mathcal{H} \ll \mathcal{H}_*$ and in particular, $\Wp{\mathcal{H}}\hookrightarrow \Lp{\mathcal{H}}$ compactly.
			\item[\textnormal{(iv)}]
			There exists a constant $C>0$ independent of $u$ such that
			\begin{align}\label{Poincare}
				\|u\|_{\mathcal{H}} \leq C\|\nabla u\|_{\mathcal{H}}\quad\text{for all
				} u \in \Wpzero{\mathcal{H}}.
			\end{align}
		\end{enumerate}
\end{proposition}
It is worth pointing out that in \cite[Proposition 2.21]{Crespo-Blanco-Gasinski-Harjulehto-Winkert-2022} the authors did not assume $q^+<N$ as well as they used a stronger condition of $\frac{q(\cdot)}{p(\cdot)}$, namely $\frac{q^+}{p^-}<1+\frac{1}{N}$. However, from the proof of \cite[Proposition 2.21]{Crespo-Blanco-Gasinski-Harjulehto-Winkert-2022}, we can easily see that it indeed needs the condition $q^+<N$ and while one can relax the condition of $\frac{q(\cdot)}{p(\cdot)}$ as mentioned above. Note that thanks to the Poincar\'e-type inequality \eqref{Poincare}, the norm $\|\nabla\cdot\|_{\mathcal{H}}$ is an equivalent norm on $\Wpzero{\mathcal{H}}$.

Regarding the critical boundary trace embedding, we have the following proposition.

\begin{proposition}\label{prop_Gen-C-T-E}
	Under hypotheses \eqref{H3} it holds that
	\begin{align*}
	\Wp{\mathcal{H}}\hookrightarrow L^{\mathcal{H}_*^{\frac{N-1}{N}}}(\Gamma)
	\end{align*}
	continuously, where
	\begin{align*}
		\mathcal{H}_*^{\frac{N-1}{N}}(x,t):=\left[\mathcal{H}_*(x,t)\right]^{\frac{N-1}{N}}
		\quad \text{for all }(x,t)\in \overline{\Omega}\times [0,\infty).
	\end{align*}
\end{proposition}
The proof of this result can be easily obtained by repeating the argument in the proof of \cite[Proposition 2.19]{Crespo-Blanco-Gasinski-Harjulehto-Winkert-2022} to verify the conditions of Theorem 4.2 in Liu-Wang-Zhao \cite{Liu-Wang-Zhao-2016}. We leave the details for the reader.

In the following, we will give an exact form of the critical terms, which will help us to study double phase problems with a larger class of nonlinear terms than in previous works.

We have the following proposition.
\begin{proposition}\label{prop_C-do-E}
	Let hypotheses  \eqref{H3} be satisfied. Then we have the continuous embedding
	\begin{align}\label{embedding-critical-domain}
		W^{1,\mathcal{H}}(\Omega)\hookrightarrow L^{\mathcal{G}^*}(\Omega),
	\end{align}
	where $\mathcal{G}^*$ is given by
	\begin{align}\label{G^*}
		\mathcal{G}^*(x,t):=t^{p^*(x)}+\mu(x)^{\frac{q^*(x)}{q(x)}}t^{q^*(x)}
		\quad \text{for }(x,t)\in \overline{\Omega}\times [0,\infty).
	\end{align}
\end{proposition}
\begin{proof}
	First, we are going to prove that
	\begin{align}\label{Est-H^*}
		\mathcal{G}^*(x,t)\leq 2\left[\left(q^*\right)^{q^*}\right]^+\mathcal{H}_*(x,t)\quad\text{for all }(x,t)\in \overline{\Omega}\times [0,\infty),
	\end{align}
	where $\mathcal{H}_*$ is the Sobolev conjugate function of $\mathcal{H}$ given in Definition \ref{def-conjugate-function}.

	For any $(x,t)\in \overline{\Omega}\times [0,\infty)$ we have $\mathcal{H}(x,t)\geq t^{p(x)}$ and so $\mathcal{H}^{-1}(x,t)\leq t^{\frac{1}{p(x)}}$. From this and \eqref{form-H_*^-1} we get
	\begin{align*}
		\mathcal{H}^{-1}_*(x,t)\leq \int^t_0 \frac{\tau^{\frac{1}{p(x)}}}{\tau^{\frac{N+1}{N}}}\,\diff \tau=p^*(x)t^{\frac{1}{p^*(x)}} \quad \text{for all }(x,t)\in \overline{\Omega}\times [0,\infty).
	\end{align*}
	It follows that
	\begin{align*}
		t=\mathcal{H}^{-1}_*(x,\mathcal{H}_*(x,t))\leq p^*(x)\left[\mathcal{H}_*(x,t)\right]^{\frac{1}{p^*(x)}} \quad \text{for all }(x,t)\in \overline{\Omega}\times [0,\infty),
	\end{align*}
	this means
	\begin{align}\label{H*-p}
		\mathcal{H}_*(x,t)\geq [p^*(x)]^{-p^*(x)}t^{p^*(x)}\quad \text{for all }(x,t)\in \overline{\Omega}\times [0,\infty).
	\end{align}

	Similarly, we obtain for any $(x,t)\in \left(\overline{\Omega}\setminus\mu^{-1}(\{0\})\right)\times [0,\infty)$ that
	\begin{align*}
		\mathcal{H}^{-1}(x,t)\leq \mu(x)^{-\frac{1}{q(x)}}t^{\frac{1}{q(x)}},
	\end{align*}
	which implies
	\begin{align*}
		\mathcal{H}^{-1}_*(x,t)\leq \int^t_0 \frac{\mu(x)^{-\frac{1}{q(x)}}\tau^{\frac{1}{q(x)}}}{\tau^{\frac{N+1}{N}}}\,\diff \tau=q^*(x)\mu(x)^{-\frac{1}{q(x)}}t^{\frac{1}{q^*(x)}}
	\end{align*}
	for all $(x,t)\in \left(\overline{\Omega}\setminus\mu^{-1}(\{0\})\right)\times [0,\infty)$. Thus we have
	\begin{align*}
		t=\mathcal{H}^{-1}_*(x,\mathcal{H}_*(x,t))\leq q^*(x)\mu(x)^{-\frac{1}{q(x)}}\left[\mathcal{H}_*(x,t)\right]^{\frac{1}{q^*(x)}}
	\end{align*}
	for all $(x,t)\in \left(\overline{\Omega}\setminus\mu^{-1}(\{0\})\right)\times [0,\infty)$.  This finally gives
	\begin{align}\label{H*-q}
		\mathcal{H}_*(x,t)\geq [q^*(x)]^{-q^*(x)}\mu(x)^{\frac{q^*(x)}{q(x)}}t^{q^*(x)}\quad \text{for all }(x,t)\in \overline{\Omega}\times [0,\infty).
	\end{align}
	From \eqref{H*-p} and \eqref{H*-q} we get that
	\begin{align*}
		[p^*(x)]^{-p^*(x)}t^{p^*(x)}+[q^*(x)]^{-q^*(x)}\mu(x)^{\frac{q^*(x)}{q(x)}}t^{q^*(x)}\leq 2\mathcal{H}_*(x,t)\quad \text{for all }(x,t)\in \overline{\Omega}\times [0,\infty).
	\end{align*}
	Then, \eqref{Est-H^*} follows.

	Invoking Propositions \ref{prop_complete}(ii) and \ref{prop_dom-emb} along with \eqref{Est-H^*} we have the continuous embeddings
	\begin{align*}
		W^{1,\mathcal{H}}(\Omega)\hookrightarrow \Lp{\mathcal{H}_*}\hookrightarrow\Lp{\mathcal{G}^*},
	\end{align*}
	which shows \eqref{embedding-critical-domain}. The proof is complete.
\end{proof}
Even the embedding $\Wp{\mathcal{H}}\hookrightarrow \Lp{\mathcal{H}_*}$ is not optimal as mentioned in the Introduction, we will try to determine the optimal Musielak-Orlicz Space $L^{\mathcal{B}_{r,s,\alpha}}(\Omega)$ among those with $\mathcal{B}_{r,s,\alpha}$ of the form
\begin{align*}
	\mathcal{B}_{r,s,\alpha}(x,t):=t^{r(x)}+\mu(x)^{\alpha(x)} t^{s(x)}
	\quad \text{for }(x,t)\in \overline{\Omega}\times [0,\infty),
\end{align*}
where $r,s,\alpha$ are positive continuous functions on $\overline{\Omega}$, such that the following continuous embedding holds
\begin{align}\label{embedding-critical-domain-O}
	W^{1,\mathcal{H}}(\Omega)\hookrightarrow L^{\mathcal{B}_{r,s,\alpha}}(\Omega).
\end{align}
By the optimal $N(\Omega)$-function $\mathcal{B}_{r_0,s_0,\alpha_0}$ for the embedding \eqref{embedding-critical-domain-O}, we mean that if  \eqref{embedding-critical-domain-O} holds for any data $(p,q,\mu,\Omega)$ satisfying the assumption \eqref{H3}, then there must be $r\leq r_0$, $s\leq s_0$ and $\alpha\geq \alpha_0$.
\begin{proposition}\label{prop_C-do-E-Op}
	Let hypotheses  \eqref{H3} be satisfied with constant exponents. Then  $\mathcal{B}_{p^*,q^*,\frac{q^*}{q}}$ is the optimal $N(\Omega)$-function for the embedding \eqref{embedding-critical-domain-O}.
\end{proposition}

\begin{proof}
	For simplification, let $\Omega=B$ be the unit ball in $\R^N$ and let $p,q,\alpha$ are constants satisfying $\alpha>1$, $1<p<q<N$ and $\frac{q}{p}<1+\frac{1}{N}$. First, we will show that if \eqref{embedding-critical-domain-O} holds then $r\leq p^*$ and $s\leq q^*$.

	Fix $u\in C_c^\infty(B)\setminus\{0\}$. For each $\lambda>0$, we define
	\begin{align*}
		v_\lambda(x):=u(\lambda x), \quad x\in\R^N.
	\end{align*}
	Clearly, $v_\lambda\in W^{1,\mathcal{H}}_0(B)$ for all $\lambda\geq 1$. If the embedding \eqref{embedding-critical-domain-O} holds, then by Proposition~\ref{prop_poincare} we find $C>0$ such that
	\begin{align}\label{P.Pro3.5-1}
		\|v_\lambda\|_{\mathcal{B}_{r,s,\alpha},\Omega}\leq C\|\nabla v_\lambda\|_{\mathcal{H}}
		\quad\text{for all }\lambda\geq 1.
	\end{align}
By the definition of the Musielak-Orlicz norm, we easily see that for $\varphi(x,t):=t^{\alpha}+w(x)t^{\beta}$ in $B\times [0,\infty)$ with $1\leq\alpha\leq\beta$, and $0\leq w(\cdot)\in L^1(B)$, it holds
\begin{align*}
\frac{1}{2}\left[\left(\int_B|v|^{\alpha}\,\diff x\right)^{\frac{1}{\alpha}}+\left(\int_Bw(x)|v|^{\beta}\,\diff x\right)^{\frac{1}{\beta}}\right]\leq \|v\|_{\varphi,B}\leq 2^{\frac{1}{\alpha}} \left[\left(\int_B|v|^{\alpha}\,\diff x\right)^{\frac{1}{\alpha}}+\left(\int_Bw(x)|v|^{\beta}\,\diff x\right)^{\frac{1}{\beta}}\right]
\end{align*}
for all $v\in L^\varphi(B)$. By means of this fact, from \eqref{P.Pro3.5-1} we find a constant $\bar{C}>0$ such that
	\begin{align}\label{O}
		\begin{split}
			&\left(\int_B|v_\lambda|^{r}\,\diff x\right)^{\frac{1}{r}}+\left(\int_B\mu(x)^\alpha|v_\lambda|^{s}\,\diff x\right)^{\frac{1}{s}}\\
			&\leq \bar{C}\left[\left(\int_B|\nabla v_\lambda|^p\,\diff x\right)^{\frac{1}{p}}+\left(\int_B\mu(x)|\nabla v_\lambda|^{q}\,\diff x\right)^{\frac{1}{q}}\right]
			\quad\text{for all }\lambda\geq 1.
		\end{split}
	\end{align}
	In order to see $s\leq q^*$, let $\mu(x)\equiv 1$. From \eqref{O} we get
	\begin{align}\label{O1}
		\left(\int_B|v_\lambda|^{s}\,\diff x\right)^{\frac{1}{s}}\leq \bar{C}\left[\left(\int_B|\nabla v_\lambda|^p\,\diff x\right)^{\frac{1}{p}}+\left(\int_B|\nabla v_\lambda|^{q}\,\diff x\right)^{\frac{1}{q}}\right]
		\quad\text{for all }\lambda\geq 1.
	\end{align}
	By change of the variable $y=\lambda x$ we get from \eqref{O1} that
	\begin{align*}
		&\lambda^{-\frac{N}{s}}\left(\int_B|u(y)|^{s}\,\diff y\right)^{\frac{1}{s}}\\
		&\leq \bar{C}\left[\lambda^{\frac{p-N}{p}}\left(\int_B|\nabla u(y)|^p\,\diff y\right)^{\frac{1}{p}}+\lambda^{\frac{q-N}{q}}\left(\int_B|y||\nabla u(y)|^{q}\,\diff y\right)^{\frac{1}{q}}\right]
		\quad\text{for all }\lambda\geq 1.
	\end{align*}
	Since the preceding inequality holds for all $\lambda\geq 1$, noticing $\frac{p-N}{p}<\frac{q-N}{q}$, we obtain
	\begin{align*}
		-\frac{N}{s}\leq \frac{q-N}{q}, \ \ \text{i.e.,}\ \ s\leq q^*.
	\end{align*}
	Next, let $\mu(x)=|x|$. Then, from \eqref{O} we have
	\begin{align*}
		\left(\int_B|v_\lambda|^{r}\,\diff x\right)^{\frac{1}{r}}\leq \bar{C}\left[\left(\int_B|\nabla v_\lambda|^p\,\diff x\right)^{\frac{1}{p}}+\left(\int_B|x||\nabla v_\lambda|^{q}\,\diff x\right)^{\frac{1}{q}}\right]
		\quad\text{for all }\lambda\geq 1,
	\end{align*}
	and
	\begin{align}\label{O3}
		\left(\int_B|x|^\alpha|v_\lambda|^{s}\,\diff x\right)^{\frac{1}{s}}\leq \bar{C}\left[\left(\int_B|\nabla v_\lambda|^p\,\diff x\right)^{\frac{1}{p}}+\left(\int_B|x||\nabla v_\lambda|^{q}\,\diff x\right)^{\frac{1}{q}}\right]
		\quad\text{for all }\lambda\geq 1.
	\end{align}
	By change of the variable $y=\lambda x$ we get from \eqref{O1} that
	\begin{align*}
		&\lambda^{-\frac{N}{r}}\left(\int_B|u(y)|^{r}\,\diff y\right)^{\frac{1}{r}}\\
		&\leq \bar{C}\left[\lambda^{\frac{p-N}{p}}\left(\int_B|\nabla u(y)|^p\,\diff y\right)^{\frac{1}{p}}+\lambda^{\frac{q-N-1}{q}}\left(\int_B|y||\nabla u(y)|^{q}\,\diff y\right)^{\frac{1}{q}}\right]
		\quad\text{for all }\lambda\geq 1.
	\end{align*}
	Since the preceding inequality holds for all $\lambda\geq 1$, noticing that $\frac{q}{p}<1+\frac{1}{N}$ is equivalent to  $\frac{p-N}{p}>\frac{q-N-1}{q}$, we obtain
	\begin{align*}
		-\frac{N}{r}\leq \frac{p-N}{p}, \ \ \text{i.e.,}\ \ r\leq p^*.
	\end{align*}
	Finally, we will show $\alpha\geq \frac{q^*}{q}$. For this purpose, let $p,q,\alpha>1$ be constants satisfying $\frac{N+1}{N}<q<N$ and
	\begin{align}\label{Re-CE-2}
		\frac{q}{p}=1+\frac{1}{N}-\eps
	\end{align}
	for $\eps\in (0,1)$ small enough. From \eqref{O3} we get
	\begin{align*}
		&\lambda^{-\frac{\alpha+N}{s}}\left(\int_B|y|^\alpha|u(y)|^{s}\,\diff y\right)^{\frac{1}{s}}\\
		&\leq \bar{C}\left[\lambda^{\frac{p-N}{p}}\left(\int_B|\nabla u(y)|^p\,\diff y\right)^{\frac{1}{p}}+\lambda^{\frac{q-N-1}{q}}\left(\int_B|y||\nabla u(y)|^{q}\,\diff y\right)^{\frac{1}{q}}\right]
		\quad\text{for all }\lambda\geq 1.
	\end{align*}
	Since the preceding inequality holds for all $\lambda\geq 1$, noticing $\frac{p-N}{p}>\frac{q-N-1}{q}$, we obtain
	\begin{align*}
		-\frac{\alpha+N}{q^*}\leq \frac{p-N}{p}.
	\end{align*}
	From this and \eqref{Re-CE-2} we easily deduce that
	\begin{align}\label{Re-CE-ep}
		\alpha\geq \frac{s}{q}-\frac{N^2}{N-q}\eps.
	\end{align}
	This means that if $\alpha< \frac{q^*}{q}$, then by taking $p\in (1,q)$ satisfying \eqref{Re-CE-2} with $\eps>0$ sufficiently small such that $\alpha< \frac{q^*}{q}-\frac{N^2}{N-q}\eps$ we have that \eqref{Re-CE-ep} cannot happen. Hence the embedding \eqref{embedding-critical-domain-O} does not hold.
	Thus, we have shown that the necessary condition for \eqref{embedding-critical-domain-O} to be valid for all $p\in (1,q)$ and for all $0\leq \mu(\cdot)\in C^{0,1}(\close)$ is $r\leq p^*$, $s\leq q^*$ and $\alpha\geq\frac{q^*}{q}$.

\end{proof}

Next, we will look for an explicit form for the critical boundary trace embedding. 
We have the following proposition.

\begin{proposition}\label{prop_C-T-E}
	Let hypotheses \eqref{H3} be satisfied. Then we have the continuous embedding
	\begin{align*}
		W^{1,\mathcal{H}}(\Omega)\hookrightarrow L^{\mathcal{T}^*}(\Gamma),
	\end{align*}
	where $\mathcal{T}^*$ is given by
	\begin{align*}
		\mathcal{T}^*(x,t):=t^{p_*(x)}+\mu(x)^{\frac{q_*(x)}{q(x)}}t^{q_*(x)}
		\quad \text{for }(x,t)\in \overline{\Omega}\times [0,\infty)
	\end{align*}
	with the critical exponents $p_*$, $q_*$ given in \eqref{critical-eponents}.
\end{proposition}

\begin{proof}
	From Jensen's inequality and \eqref{Est-H^*}, we have
	\begin{align*}
		\left[\mathcal{T}^*(x,t)\right]^{\frac{N}{N-1}}
		\leq 2^{\frac{N}{N-1}-1}\mathcal{G}^*(x,t)
		\leq 2^{\frac{N}{N-1}}\left[\left(q^*\right)^{q^*}\right]^+\mathcal{H}_*(x,t)
		\quad\text{for all } (x,t)\in \overline{\Omega}\times [0,\infty).
	\end{align*}
	This  implies that
	\begin{align}\label{Est-T^*}
		\mathcal{T}^*(x,t)\leq 2\left(\left[\left(q^*\right)^{q^*}\right]^+\right)^{\frac{N-1}{N}}\left[\mathcal{H}_*(x,t)\right]^{\frac{N-1}{N}}\quad\text{for all } (x,t)\in \overline{\Omega}\times [0,\infty).
	\end{align}

	Let $u\in W^{1,\mathcal{H}}(\Omega)$. From Proposition \ref{prop_Gen-C-T-E} we have $u\in L^{\mathcal{W}}(\Gamma)$, where $\mathcal{W}(x,t):= \left[\mathcal{H}_*(x,t)\right]^{\frac{N-1}{N}}$. Hence, $u\in L^{\mathcal{T}^*}(\Gamma)$ due to \eqref{Est-T^*}. We set $\lambda=\|u\|_{\mathcal{W},\Gamma}$ and assume first that $\lambda>0$. Then, we obtain
	\begin{align*}
		1=\int_\Gamma \mathcal{W}\left(x,\left|\frac{u}{\lambda}\right|\right)\,\diff\sigma\geq c_0^{-1} \int_\Gamma \mathcal{T}^*\left(x,\left|\frac{u}{\lambda}\right|\right)\,\diff\sigma,
	\end{align*}
	where $c_0:=2\left(\left[\left(q^*\right)^{q^*}\right]^+\right)^{\frac{N-1}{N}}$. This implies
	\begin{align*}
		1\geq \int_\Gamma \mathcal{T}^*\left(x,\left|\frac{u}{c_0^{\frac{1}{(q_*)^+}}\lambda}\right|\right)\,\diff\sigma.
	\end{align*}
	Hence, we get
	\begin{align*}
		\|u\|_{\mathcal{T}^*,\Gamma}\leq c_0^{\frac{1}{(q_*)^+}}\lambda=c_0^{\frac{1}{(q_*)^+}}\|u\|_{\mathcal{W},\Gamma}.
	\end{align*}
	From this and Proposition \ref{prop_Gen-C-T-E} we arrive at
	\begin{align*}
		\|u\|_{\mathcal{T}^*,\Gamma}\leq C\|u\|_{1,\mathcal{H},\Omega},
	\end{align*}
	where $C$ is a positive constant independent of $u$. The proof is complete.
\end{proof}

In the last part of this section we prove new compact embedding results.

\begin{proposition}\label{prop_S-C-E}
	Let hypotheses \eqref{H3} be satisfied and let
	\begin{align*}
		\Psi(x,t):=t^{r(x)}+\mu(x)^{\frac{s(x)}{q(x)}}t^{s(x)}
		\quad\text{for } (x,t)\in \overline{\Omega}\times [0,\infty),
	\end{align*}
	  where $r,s\in C(\close)$ satisfy $1<r(x)\leq p^*(x)$ and $1<s(x)\leq q^*(x)$ for all $x\in\overline{\Omega}$. Then, we have the continuous embedding
	\begin{align}\label{Prop-S-E}
		W^{1,\mathcal{H}}(\Omega)
		\hookrightarrow  L^{\Psi}(\Omega).
	\end{align}
	Furthermore, if $r(x)<p^*(x)$ and $s(x)< q^*(x)$ for all $x\in\overline{\Omega}$, then the embedding in \eqref{Prop-S-E} is compact.
\end{proposition}

\begin{proof}
	First, it is clear that
	\begin{align*}
		\Psi(x,t)
		\leq t^{p^*(x)}+1+\mu(x)^{\frac{q^*(x)}{q(x)}}t^{q^*(x)}+1=\mathcal{G}^*(x,t)+2\quad\text{for all }(x,t)\in\overline{\Omega} \times [0,\infty).
	\end{align*}
	From this along with Propositions \ref{prop_complete} and \ref{prop_C-do-E} we obtain \eqref{Prop-S-E}.

	Let us now suppose that $r(x)<p^*(x)$ and $s(x)< q^*(x)$ for all $x\in\overline{\Omega}$. In order to prove the compactness of the embedding in \eqref{Prop-S-E} it is sufficient to show that $\Psi \ll \mathcal{H}_*$ due to Proposition \ref{prop_dom-emb}(ii). This means, for any $k>0$, we need to show that
	\begin{align}\label{Pro-S-E-1}
		\lim_{t\to\infty} \frac{\Psi(x,kt)}{\mathcal{H}_*(x,t)}=0 \quad \text{uniformly for a.\,a.\,}x\in\Omega.
	\end{align}

	Indeed, from  \eqref{Est-H^*} we have for $(x,t,k)\in \Omega\times [0,\infty)\times (0,\infty)$ the estimate
	\begin{align*}
		\frac{\Psi(x,kt)}{\mathcal{H}_*(x,t)}
		&\leq 2\left[\left(q^*\right)^{q^*}\right]^+\frac{k^{r(x)}t^{r(x)}+k^{s(x)}\mu(x)^{\frac{s(x)}{q(x)}}t^{s(x)}}{t^{p^*(x)}+\mu(x)^{\frac{q^*(x)}{q(x)}}t^{q^*(x)}}\\
		&\leq 2\left[\left(q^*\right)^{q^*}\right]^+\left(1+k^{r^+}+k^{s^+}\right)\frac{t^{r(x)}+\mu(x)^{\frac{s(x)}{q(x)}}t^{s(x)}}{t^{p^*(x)}+\mu(x)^{\frac{q^*(x)}{q(x)}}t^{q^*(x)}}.
	\end{align*}
	Then, by using Young's inequality with $\eps>0$, we obtain
	\begin{align*}
		t^{r(x)}\leq \eps t^{p^*(x)}+\eps^{-\frac{r(x)}{p^*(x)-r(x)}}\leq \eps t^{p^*(x)}+\eps^{-\left(\frac{r}{p^*-r}\right)^+}+1
	\end{align*}
	and
	\begin{align*}
		\mu(x)^{\frac{s(x)}{q(x)}}t^{s(x)}\leq \eps \mu(x)^{\frac{q^*(x)}{q(x)}}t^{q^*(x)} +\eps^{-\frac{s(x)}{q^*(x)-s(x)}}\leq \eps \mu(x)^{\frac{q^*(x)}{q(x)}}t^{q^*(x)}+\eps^{-\left(\frac{s}{q^*-r}\right)^+}+1.
	\end{align*}
	Combining the last three estimates, we easily get \eqref{Pro-S-E-1} and this completes the proof.
\end{proof}

Similarly to Proposition \ref{prop_S-C-E}, we have the following compact boundary trace embedding.

\begin{proposition}\label{prop_S-C-T-E}
	Let hypotheses \eqref{H3} be satisfied and let
	\begin{align*}
		\Upsilon (x,t)=t^{\ell(x)}+\mu(x)^{\frac{m(x)}{q(x)}}t^{m(x)}
		\quad\text{for } (x,t)\in \overline{\Omega}\times [0,\infty),
	\end{align*}
	where $\ell,m\in C(\close)$ satisfy $1<\ell(x)\leq p_*(x)$ and $1<m(x)\leq q_*(x)$ for all $x\in\overline{\Omega}$. Then, we have the continuous embedding
	\begin{align}\label{c-T}
		W^{1,\mathcal{H}}(\Omega)\hookrightarrow L^{\Upsilon }(\Gamma).
	\end{align}
	Furthermore, if $\ell(x)<p_*(x)$ and $m(x)<q_*(x)$ for all $x\in\overline{\Omega}$, then the embedding \eqref{c-T} is compact.
\end{proposition}

\begin{proof}
	We have
	\begin{align}\label{est-1}
		\Upsilon(x,t)
		\leq t^{p_*(x)}+1+\mu(x)^{\frac{q_*(x)}{q(x)}}t^{q_*(x)}+1=\mathcal{T}^*(x,t)+2\quad\text{for all }(x,t)\in\close \times [0,\infty).
	\end{align}
	Then, the embedding \eqref{c-T} follows from Propositions \ref{prop_complete} and \ref{prop_C-T-E} by taking \eqref{est-1} into account.

	Next, suppose that $\ell(x)<p_*(x)$ and $m(x)<q_*(x)$ for all $x\in\overline{\Omega}$ and note that
	\begin{align}\label{P.C-T.1}
		W^{1,\mathcal{H}}(\Omega)\hookrightarrow W^{1,p^-}(\Omega)\hookrightarrow \hookrightarrow L^{1}(\Gamma).
	\end{align}
	Let $\{u_n\}_{n\in\N}$ be a bounded sequence in $W^{1,\mathcal{H}}(\Omega)$. From \eqref{P.C-T.1} we can suppose, up to a subsequence not relabeled, that $u_n\to u$ in measure on $\Gamma$. Let $\eps>0$ be given and set
	\begin{align*}
		v_{j,k}(x):=\frac{u_j(x)-u_k(x)}{\eps}\quad \text{for } j,k\in\mathbb{N}.
	\end{align*}
	From Proposition \ref{prop_C-T-E} we see that $\{v_{j,k}\}_{j,k\in\N}$ is bounded in $L^{\mathcal{T}^*}(\Gamma)$, say
	$\|v_{j,k}\|_{\mathcal{T}^*,\Gamma}\leq k_0$ for all $j,k\in\mathbb{N}$. Arguing as in the proof of Proposition \ref{prop_S-C-E}, we find $t_0>0$ such that
	\begin{align*}
		\Upsilon(x,t)\leq \frac{1}{4}	\mathcal{T}^*\l(x,\frac{t}{k_0}\r)
		\quad \text{for all } (x,t)\in \overline{\Omega}\times (t_0,\infty).
	\end{align*}
	Then, arguing as in the proof of Theorem 8.24 of Adams-Fournier \cite{Adams-Fournier-2003}, we find $N_\eps$ such that $\|v_{j,k}\|_{\Upsilon,\Gamma}<1$ for all $j,k\geq N_\eps$. That is, we have shown that $\|u_j-u_k\|_{\Upsilon,\Gamma}<\eps$ for all $j,k\geq N_\eps$. Hence, $u_n\to u$ in $L^{\Upsilon}(\Gamma)$. The proof is complete.
\end{proof}

\section{A priori bounds for generalized double phase problems with subcritical growth}\label{Section-4}

In this section, we prove the boundedness of weak solutions to the problems \eqref{D} and \eqref{N} when the nonlinearities involved satisfy a subcritical growth as developed in Propositions \ref{prop_S-C-E} and \ref{prop_S-C-T-E}. The proofs are using ideas from the papers of Ho-Kim \cite{Ho-Kim-2019}, Ho-Kim-Winkert-Zhang \cite{Ho-Kim-Winkert-Zhang-2022} and Winkert-Zacher \cite{Winkert-Zacher-2012,Winkert-Zacher-2015}.

Let hypotheses \eqref{H3} be satisfied. We suppose the following structure conditions on $\mathcal{A}$ and $\mathcal{B}$:
\begin{enumerate}[label=\textnormal{(D$_1$):},ref=\textnormal{D$_1$}]
	\item\label{D1}
		The functions $\mathcal{A}\colon\Omega\times\R\times\R^N\to \R^N$ and $\mathcal{B}\colon\Omega \times \R\times \R^N\to \R$ are Carath\'eodory functions such that
		\begin{enumerate}[label=\textnormal{(\roman*)},ref=\textnormal{(D$_1$)(\roman*)}]
			\item
				$ |\mathcal{A}(x,t,\xi)|
				\leq \alpha_1 \left[|t|^{\frac{p^*(x)}{p'(x)}}+\mu(x)^{\frac{N-1}{N-q(x)}}|t|^{\frac{q^*(x)}{q'(x)}}+|\xi|^{p(x)-1} +\mu(x)|\xi|^{q(x)-1}+1\right],$
			\item\label{D1ii}
				$ \mathcal{A}(x,t,\xi)\cdot \xi \geq \alpha_2 \left[|\xi|^{p(x)} +\mu(x)|\xi|^{q(x)}\right]-\alpha_3\left[|t|^{r(x)}+\mu(x)^{\frac{s(x)}{q(x)}}|t|^{s(x)}+1\right],$
			\item\label{D1iii}
				$ |\mathcal{B}(x,t,\xi)|
				\leq \beta \left[|t|^{r(x)-1}+\mu(x)^{\frac{s(x)}{q(x)}}|t|^{s(x)-1}+|\xi|^{\frac{p(x)}{r'(x)}} +\mu(x)^{\frac{1}{q(x)}+\frac{1}{s'(x)}}|\xi|^{\frac{q(x)}{s'(x)}}+1 \right],$
		\end{enumerate}
		for a.\,a.\,$x\in\Omega$ and for all $(t,\xi) \in \R\times\R^N$, where $\alpha_1,\alpha_2,\alpha_3$, $\beta$ are positive constants and $r,s\in C(\overline{\Omega})$ satisfy $p(x)<r(x)<p^*(x)$ and $q(x)<s(x)<q^*(x)$ for all $x\in\overline{\Omega}$.
\end{enumerate}

For the second problem \eqref{N} with nonlinear boundary condition we need the additional assumption on $\mathcal{C}$:

\begin{enumerate}[label=\textnormal{(N$_1$):},ref=\textnormal{N$_1$}]
	\item\label{N1}
		The function $\mathcal{C}\colon\Gamma \times \R\to \R$ is a Carath\'eodory function such that
	\begin{enumerate}[label=\textnormal{(\roman*)},ref=\textnormal{(N$_1$)(\roman*)}]
		\item[]
			$ |\mathcal{C}(x,t)|\leq \gamma \left[|t|^{\ell(x)-1}+\mu(x)^{\frac{h(x)}{q(x)}}|t|^{h(x)-1}+1\right]$
	\end{enumerate}
	for a.\,a.\,$x\in\Gamma$ and for all $t\in \mathbb{R}$, where $\gamma$ is a positive constant and $\ell,h\in C(\overline{\Omega})$ satisfy $p(x)<\ell(x)<p_*(x)$ and $q(x)<h(x)<q_*(x)$ for all $x\in\overline{\Omega}$.
\end{enumerate}

The weak formulation of \eqref{D} and \eqref{N} read as follows.

\begin{definition}\label{Def.Sol} $~$
	\begin{enumerate}
		\item[\textnormal{(i)}]
			We say that $ u\in W_0^{1,\mathcal{H}}(\Omega) $ is a weak solution of problem \eqref {D}  if
			\begin{align}\label{def_sol_D}
				\int_{\Omega}\mathcal{A}(x,u,\nabla u)\cdot \nabla \varphi \,\diff x
				= \int_{\Omega}\mathcal{B}(x,u, \nabla u)\varphi \,\diff x
			\end{align}
			is satisfied for all $\ph \in W_0^{1,\mathcal{H}}(\Omega)$.
		\item[\textnormal{(ii)}]
			We say that $ u\in W^{1,\mathcal{H}}(\Omega) $ is a weak solution of problem \eqref {N} if
			\begin{align}\label{def_sol_N}
				\int_{\Omega}\mathcal{A}(x,u,\nabla u)\cdot \nabla \varphi \,\diff x
				= \int_{\Omega}\mathcal{B}(x,u, \nabla u)\varphi \,\diff x+\int_{\Gamma}\mathcal{C}(x,u)\varphi \,\diff\sigma
			\end{align}
			is satisfied for all $\varphi\in W^{1,\mathcal{H}}(\Omega)$.
	\end{enumerate}
\end{definition}

Under the assumptions \eqref{D1} and \eqref{N1} we know that the terms in \eqref{def_sol_D} and \eqref{def_sol_N} are well-defined due to Propositions  \ref{prop_S-C-E} and \ref{prop_S-C-T-E}.

For the Dirichlet problem \eqref{D} we have the following result.

\begin{theorem}\label{D.a-priori}
	Let hypotheses \eqref{H3} and \eqref{D1} be satisfied. Then, any weak solution $u\in \Wpzero{\mathcal{H}}$ of problem \eqref{D} belongs to $L^\infty(\Omega)$ and satisfies the following a priori estimate
	\begin{align}\label{D-bound}
		\|u\|_{\infty,\Omega} \leq C \max \big\{\|u\|_{\Psi,\Omega}^{\tau_1},\|u\|_{\Psi,\Omega}^{\tau_2} \big\},
	\end{align}
	where $C,\tau_1,\tau_2$ are positive constants independent of $u$ and
	\begin{align*}
		\Psi(x,t):=t^{r(x)}+\mu(x)^{\frac{s(x)}{q(x)}}t^{s(x)}
		\quad\text{for } (x,t)\in \close\times [0,\infty).
	\end{align*}
\end{theorem}

\begin{proof}
	The proof is based on the ideas used in Ho-Kim \cite{Ho-Kim-2019}, Winkert-Zacher \cite{Winkert-Zacher-2012,Winkert-Zacher-2015} and will use Lemma \ref{leRecur}. Let $u$ be a weak solution of problem \eqref{D}.

	\vskip5pt
	{\bf Step 1.\,Defining the recursion sequence and basic estimates.}

	For each $n\in\N_0$, we define
	\begin{align*}
		Z_n:=\int_{A_{\kappa_n} }\left[(u-\kappa_n)^{r(x)}+\mu(x)^{\frac{s(x)}{q(x)}}(u-\kappa_n)^{s(x)}\right]\,\diff x,
	\end{align*}
	where
	\begin{align}\label{def.Ak}
		A_\kappa:=\{x\in\Omega\,:\,\ u(x)>\kappa\}, \quad  \kappa\in\R
	\end{align}
	and
	\begin{align}\label{def.kn}
		\kappa_n:=\kappa_*\left(2-\frac{1}{2^n}\right), \quad n\in \mathbb{N}_0
	\end{align}
	with $\kappa_*>0$ to be specified later. Obviously,
	\begin{align*}
		\kappa_n \nearrow 2\kappa_*
		\quad \text{and}\quad
		\kappa_* \leq \kappa_n <2\kappa_* \quad \text{for all }n\in \mathbb{N}_0.
	\end{align*}
	It is clear that
	\begin{align}\label{S-Zn.decreasing}
		A_{\kappa_{n+1}}\subset A_{\kappa_n}
		\quad \text{and}\quad
		Z_{n+1}\leq Z_n \quad\text{for all }n\in \mathbb{N}_0.
	\end{align}
	Moreover, from the estimates
	\begin{align*}
		u(x)- \kappa_n \geq u(x)\l(1-\frac{\kappa_n}{\kappa_{n+1}}\r) = \frac{u(x)}{2^{n+2}-1} \quad \text{for a.\,a.\,} x \in A_{\kappa_{n+1}}
	\end{align*}
	and
	\begin{align*}
		|A_{\kappa_{n+1}}|
		\leq \int_{A_{\kappa_{n+1}}} \left (\frac{u-\kappa_n}{\kappa_{n+1}-\kappa_n}\right )^{r(x)} \,\diff x
		\leq \int_{A_{\kappa_{n}}} \frac{2^{r(x)(n+1)}}{\kappa_*^{r(x)}} (u-\kappa_n)^{r(x)}\,\diff x,
	\end{align*}
	we obtain the following inequalities
	\begin{align} \label{S-est.u}
		u(x)\leq (2^{n+2}-1)(u(x)-\kappa_n)
		\quad \text{for a.\,a.\,}x\in A_{\kappa_{n+1}}\quad \text{for all } n\in\N_0
	\end{align}
	and
	\begin{align} \label{S-|A_{k_{n+1}}|}
		|A_{\kappa_{n+1}}| \leq \left(\kappa_{*}^{-r^{-}}+\kappa_{*}^{-r^{+}}\right)2^{(n+1)r^+} Z_n\leq 2\left(1+\kappa_{*}^{-r^{+}}\right)2^{(n+1)r^+} Z_n\quad \text{for all } n\in\N_0.
	\end{align}
	By the assumptions on the exponents we have
	\begin{align*}
		&\int_{A_{\kappa_{n+1}}}\left[(u-\kappa_{n+1})^{p(x)}+\mu(x)(u-\kappa_{n+1})^{q(x)}\right]\,\diff x\\
		&\leq \int_{A_{\kappa_{n+1}}}\left[(u-\kappa_{n+1})^{r(x)}+\mu(x)^{\frac{s(x)}{q(x)}}(u-\kappa_{n+1})^{s(x)}+2\right]\,\diff x.
	\end{align*}
	Combining this with \eqref{S-Zn.decreasing} and \eqref{S-|A_{k_{n+1}}|} gives
	\begin{align}\label{S-D.u}
		\int_{A_{\kappa_{n+1}}}\left[(u-\kappa_{n+1})^{p(x)}+\mu(x)(u-\kappa_{n+1})^{q(x)}\right]\,\diff x\leq 5\left(1+\kappa_{*}^{-r^{+}}\right)2^{(n+1)r^+} Z_n
	\end{align}
	for all $n\in\N_0$.

	Next, we are going to show the following estimate for truncated energies
	\begin{align}\label{S-D.grad}
		\int_{A_{\kappa_{n+1}}}\left[|\nabla u|^{p(x)}+\mu(x)|\nabla u|^{q(x)}\right]\,\diff x\leq  C_1(1+\kappa_{*}^{-r^+})2^{n(r^++s^+)}Z_n
		\quad\text{for all } n\in\N_0.
	\end{align}
	Here and in the rest of the proof, $C_i$ ($i\in\N$) are positive constants independent of $u$, $n$ and $\kappa_{*}$. To this end, testing \eqref{def_sol_D} by $\varphi =(u-\kappa_{n+1})_{+} \in W_0^{1,\mathcal{H}}(\Omega)$ gives
	\begin{align}\label{S-D.var.Eq}
		\int_{A_{\kappa_{n+1}}}
		\mathcal{A}(x,u,\nabla u) \cdot \nabla u \,\diff x =\int_{A_{\kappa_{n+1}} }\mathcal{B}(x,u,\nabla u)(u-\kappa_{n+1})\,\,\diff x.
	\end{align}
	Since $u\geq u-\kappa_{n+1} >0$ on $A_{\kappa_{n+1}},$ using \ref{D1ii} and \ref{D1iii} along with Young's inequality and the fact that $u\leq u^{r(x)}+1$ on $A_{\kappa_{n+1}}$, we obtain the estimates
	\begin{align*}
		&\int_{A_{\kappa_{n+1}}} \mathcal{A}(x,u,\nabla u)\cdot \nabla u\,\diff x \\
		&\geq \alpha_2\int_{A_{\kappa_{n+1}} }\left[|\nabla u|^{p(x)}+\mu(x)|\nabla u|^{q(x)}\right]\,\diff x-\alpha_3\int_{A_{\kappa_{n+1}} }\left[u^{r(x)}+\mu(x)^{\frac{s(x)}{q(x)}}u^{s(x)}+1\right]\,\diff x
	\end{align*}
	and
	\begin{align*}
		&\int_{A_{\kappa_{n+1}}} \mathcal{B}(x,u,\nabla 	u)(u-\kappa_{n+1})\,\diff x\\
		&\leq \beta\int_{A_{\kappa_{n+1}} }\left(u^{r(x)-1}+\mu(x)^{\frac{s(x)}{q(x)}}u^{s(x)-1}+|\nabla u|^{\frac{p(x)}{r'(x)}} +\mu(x)^{\frac{1}{q(x)}+\frac{1}{s'(x)}}|\nabla u|^{\frac{q(x)}{s'(x)}}+1\right)u\,\diff x\\
		&\leq \frac{\alpha_2}{2}\int_{A_{\kappa_{n+1}} }\left[|\nabla u|^{p(x)}+\mu(x)|\nabla u|^{q(x)}\right]\,\diff x+C_2\int_{A_{\kappa_{n+1}} }\left[u^{r(x)}+\mu(x)^{\frac{s(x)}{q(x)}}u^{s(x)}+1\right]\,\diff x.
	\end{align*}
	Combining the last two estimates with \eqref{S-D.var.Eq} and then using \eqref{S-est.u} it follows
	\begin{align*}
		&\int_{A_{\kappa_{n+1}} }\left[|\nabla u|^{p(x)}+\mu(x)|\nabla u|^{q(x)}\right]\,\diff x\\
		&\leq C_3\int_{A_{\kappa_{n+1}} }\left[u^{r(x)}+\mu(x)^{\frac{s(x)}{q(x)}}u^{s(x)}+1\right]\,\diff x\\
		&\leq C_3\int_{A_{\kappa_{n+1}}}\left(\left[(2^{n+2}-1)(u-\kappa_{n})\right]^{r(x)}+\mu(x)^{\frac{s(x)}{q(x)}}\left[(2^{n+2}-1)(u-\kappa_{n})\right]^{s(x)}\right)\,\diff x+C_3|A_{\kappa_{n+1}}|\\
		&\leq C_42^{n(r^++s^+})\int_{A_{\kappa_{n}}}\left[(u-\kappa_{n})^{r(x)}+\mu(x)^{\frac{s(x)}{q(x)}}(u-\kappa_{n})^{s(x)}\right]\,\diff x+C_3|A_{\kappa_{n+1}}|.
	\end{align*}
	From this and \eqref{S-|A_{k_{n+1}}|} we obtain \eqref{S-D.grad}.

	\vskip5pt
	\textbf{Step 2. Estimating $Z_{n+1}$ by $Z_n$.}

	In the following, we estimate $Z_{n+1}$ by $Z_n$ with $n\in\N_0$. To this end, let $\{B_i\}_{i=1}^m$ be a finite open covering of $\overline{\Omega}$, where $B_i$ ($i\in \{1,\cdots,m\}$) are open balls of radius $R$ in $\R^N$ such that $\Omega_i:=B_i\cap\Omega$ ($i\in \{1,\cdots,m\}$) are Lipschitz domains. We may take $R$ sufficiently small such that
	\begin{align}\label{S-D.Omega_i}
		|\Omega_{i}|<1\quad \text{for all } i\in\{1,\cdots,m\},
	\end{align}
	and
	\begin{align}\label{S-D.loc.exp}
		p_i^+<r^{-}_{i}\leq r_i^+<\left(p^*\right)^-_i\ \ \text{and} \ \ q^+<s_i^-\leq s_i^+<\left(q^*\right)^-_i\quad \text{for all } i\in\{1,\cdots,m\},
	\end{align}
	where for a function $f\in C\left(\overline{\Omega}\right)$ and $i\in\{1,\cdots,m\}$, we denote
	\begin{align*}
		f_i^+:=\max_{x\in\overline{\Omega}_i} f(x)
		\quad \text{and} \quad f^{-}_{i}:=\min_{x\in\overline{\Omega}_i} f(x).
	\end{align*}
	Let $n\in\N_0$ and denote $v_n:=(u-\kappa_{n+1})_+$. For each $i\in\{1,\cdots,m\}$, $\hat{\alpha}>0$, and $\hat{\beta}>0$, we denote
	\begin{align*}
		T_{n,i}(\hat{\alpha},\hat{\beta}):=\int_{\Omega_{i}}\left[v_n^{\hat{\alpha}}+\mu(x)^{\frac{\hat{\beta}}{q(x)}}v_n^{\hat{\beta}}\right]\,\diff x.
	\end{align*}
	We have
	\begin{align*}
		Z_{n+1}=\int_{\Omega}\left[v_n^{r(x)}+\mu(x)^{\frac{s(x)}{q(x)}}v_n^{s(x)}\right]\,\diff x
		\leq \sum_{i=1}^m\int_{\Omega_{i}}\left[v_n^{r(x)}+\mu(x)^{\frac{s(x)}{q(x)}}v_n^{s(x)}\right]\,\diff x.
	\end{align*}
	From this and the basic inequality
	\begin{align}\label{basic.ineq}
		t^{\tilde{\beta}}\leq t^{\tilde{\alpha}}+t^{\tilde{\gamma}}
		\quad \text{for all } t\geq 0 \text{ and for all } \tilde{\alpha}, \tilde{\beta}, \tilde{\gamma} \text{ with } 0<\tilde{\alpha}\leq\tilde{\beta}\leq\tilde{\gamma},
	\end{align}
	we obtain
	\begin{align}\label{S-decompose1}
		Z_{n+1}\leq \sum_{i=1}^m\big[T_{n,i}(r_i^-,s_i^-)+T_{n,i}(r_i^+,s_i^+)\big].
	\end{align}
	From \eqref{S-D.loc.exp}, we can fix $\eps$ such that
	\begin{align}\label{S-D-eps}
		0<\eps<\min_{1\leq i\leq m}\ \min\{\left(p^*\right)^-_i-r_i^+, \left(q^*\right)^-_i-s_i^+\}.
	\end{align}
	Let $i\in\{1,\cdots,m\}$ and let $\star\in\{+,-\}$. By H\"older's inequality and \eqref{S-D.Omega_i} we have
	\begin{align}\label{S-de-E1}
		\begin{split}
			T_{n,i}(r_i^\star,s_i^\star)&=\int_{A_{\kappa_{n+1}}\cap \Omega_i}\left[v_n^{r_i^\star}+\mu(x)^{\frac{s_i^\star}{q(x)}}v_n^{s_i^\star}\right]\,\diff x\\
			&\leq \left(\int_{\Omega_{i}}v_n^{r_i^\star+\eps}\,\diff x\right)^{\frac{r_i^\star}{r_i^\star+\eps}}|A_{\kappa_{n+1}}\cap \Omega_i|^{\frac{\eps}{r_i^\star+\eps}}\\
			&\quad +\left(\int_{\Omega_{i}}\mu(x)^{\frac{s_i^\star+\eps}{q(x)}}v_n^{s_i^\star+\eps}\,\diff x\right)^{\frac{s_i^\star}{s_i^\star+\eps}}|A_{\kappa_{n+1}}\cap \Omega_i|^{\frac{\eps}{s_i^\star+\eps}}\\
			&\leq |A_{\kappa_{n+1}}\cap \Omega_i|^{\frac{\eps}{r^++s^++\eps}}\left[\left(\int_{\Omega_{i}}v_n^{r_i^\star+\eps}\,\diff x\right)^{\frac{r_i^\star}{r_i^\star+\eps}}+\left(\int_{\Omega_{i}}\mu(x)^{\frac{s_i^\star+\eps}{q(x)}}v_n^{s_i^\star+\eps}\,\diff x\right)^{\frac{s_i^\star}{s_i^\star+\eps}}\right].
		\end{split}
	\end{align}
	Denote
	\begin{align*}
		\Psi_\star(x,t):=t^{r_i^\star+\eps}+\mu(x)^{\frac{s_i^\star+\eps}{q(x)}}t^{s_i^\star+\eps}.
	\end{align*}
	By \eqref{S-D-eps}, it holds
	\begin{align*}
		r_i^\star+\eps<\left(p^*\right)^-_i
		\quad \text{and}\quad s_i^\star+\eps<\left(q^*\right)^-_i.
	\end{align*}
	Hence, we have
	\begin{align}\label{S-loc-emb'}
	W^{1,p(\cdot)}(\Omega_i)\hookrightarrow	W^{1,p_i^-}(\Omega_i)\hookrightarrow L^{	r_i^\star+\eps}(\Omega_i)
	\end{align}
and
\begin{align}\label{S-loc-emb}
 W^{1,\mathcal{H}}(\Omega_i)\hookrightarrow L^{\Psi_\star}(\Omega_i)
\end{align}
in view of Proposition \ref{proposition_embeddings} (for the case $\Omega=\Omega_i$) and Proposition \ref{prop_S-C-E} (for the case $\Omega=\Omega_i$), respectively. Taking the embedding \eqref{S-loc-emb'} and Proposition \ref{prop_mod-nor2} (for the case $\mu\equiv 0$ and $\Omega=\Omega_i$) into account we have
	\begin{align}\label{S-de-E2}
		\left(\int_{\Omega_{i}}v_n^{r_i^\star+\eps}\,\diff x\right)^{\frac{r_i^\star}{r_i^\star+\eps}}
			\leq C_5\|v_n\|_{1,p(\cdot),\Omega_i}^{r_i^\star}\leq C_5\left(R_{n,i}^{\frac{r_i^\star}{p_i^-}}+R_{n,i}^{\frac{r_i^\star}{p_i^+}} \right),
	\end{align}
	where
\begin{align*}
	R_{n,i}:=\int_{\Omega_i}\left[|\nabla v_n|^{p(x)}+\mu(x)|\nabla v_n|^{q(x)}\right]\,\diff x + \int_{\Omega_i}\left[ v_n^{p(x)}+\mu(x)v_n^{q(x)}\right]\,\diff x.
\end{align*}
On the other hand, by invoking the embedding \eqref{S-loc-emb} and Proposition \ref{prop_mod-nor2} we have
\begin{align}\label{S-de-E2'}
\left(\int_{\Omega_{i}}\mu(x)^{\frac{s_i^\star+\eps}{q(x)}}v_n^{s_i^\star+\eps}\,\diff x\right)^{\frac{s_i^\star}{s_i^\star+\eps}}
	\leq \|v_n\|_{\Psi_\star,\Omega_i}^{s_i^\star}\leq C_6\|v_n\|_{1,\mathcal{H},\Omega_i}^{s_i^\star}\leq C_6\left(R_{n,i}^{\frac{s_i^\star}{p_i^-}}+R_{n,i}^{\frac{s_i^\star}{q_i^+}} \right).
\end{align}
	From \eqref{S-de-E1}, \eqref{S-de-E2} and \eqref{S-de-E2'}, we obtain
	\begin{align}\label{S-de-E3}
		T_{n,i}(r_i^\star,s_i^\star)
		\leq C_7|A_{\kappa_{n+1}}\cap \Omega_i|^{\frac{\eps}{r^++s^++\eps}}\left(R_{n,i}^{\frac{r_i^\star}{p_i^-}}+R_{n,i}^{\frac{r_i^\star}{p_i^+}}+R_{n,i}^{\frac{s_i^\star}{p_i^-}}+R_{n,i}^{\frac{s_i^\star}{q_i^+}}\right).
	\end{align}
	Invoking \eqref{S-de-E3} and  \eqref{basic.ineq} we infer
	\begin{align}\label{S-de-E5}
		T_{n,i}(r_i^\star,s_i^\star)\leq C_8|A_{\kappa_{n+1}}|^{\frac{\eps}{r^++s^++\eps}}\left(R_{n}^{1+\gamma_1}+R_{n}^{1+\gamma_2}\right),
	\end{align}
	where
	\begin{align*}
		R_{n}:=\int_{\Omega}\left[|\nabla v_n|^{p(x)}+\mu(x)|\nabla v_n|^{q(x)}\right]\,\diff x	+\int_{\Omega}\left[ v_n^{p(x)}+\mu(x)v_n^{q(x)}\right]\,\diff x
	\end{align*}
	and
	\begin{align*}
		0<\gamma_1:=\min_{1\leq i\leq m} \min\left\{\frac{r_i^-}{p_i^+},\frac{s_i^-}{q_i^+}\right\}-1\leq\gamma_2:=\max_{1\leq i\leq m} \max\left\{\frac{r_i^+}{p_i^-},\frac{s_i^+}{p_i^-}\right\}-1.
	\end{align*}
	Using the estimate \eqref{S-de-E5}, we deduce from \eqref{S-decompose1} that
	\begin{align}\label{S-de-E6}
		Z_{n+1}\leq 	C_9|A_{\kappa_{n+1}}|^{\frac{\eps}{r^++s^++\eps}}\left(R_{n}^{1+\gamma_1}+R_{n}^{1+\gamma_2}\right).
	\end{align}
	On the other hand, combining \eqref{S-D.u} with \eqref{S-D.grad} gives
	\begin{align*}
		R_n\leq C_{10}\left(1+\kappa_{*}^{-r^+}\right)2^{n(r^++s^+)}Z_n\quad \text{for all } n\in\N_0.
	\end{align*}
	Thus,
	\begin{align}\label{S-de-E7}
		R_{n}^{1+\gamma_1}+R_{n}^{1+\gamma_2}\leq C_{11}\left(1+\kappa_{*}^{-r^+(1+\gamma_2)}\right)2^{n(r^++s^+)(1+\gamma_2)}\left(Z_{n}^{1+\gamma_1}+Z_{n}^{1+\gamma_2}\right).
	\end{align}
	Moreover, \eqref{S-|A_{k_{n+1}}|} implies that
	\begin{align}\label{S-de-E8}
		|A_{\kappa_{n+1}}|^{\frac{\eps}{r^++s^++\eps}}\leq C_{12}\left(\kappa_{*}^{-\frac{\eps r^{-}}{r^++s^++\eps}}+\kappa_{*}^{-\frac{\eps r^{+}}{r^++s^++\eps}}\right)2^{\frac{\eps r^+}{r^++s^++\eps}n} Z_n^{\frac{\eps}{r^++s^++\eps}}.
	\end{align}
	From \eqref{S-de-E6}, \eqref{S-de-E7} and \eqref{S-de-E8} along with \eqref{basic.ineq} we arrive at
	\begin{align}\label{S-D-ReIneq}
		Z_{n+1}
		\leq
		C_{13}\round{\kappa_{*}^{-\mu_1}+\kappa_{*}^{-\mu_2}}b^n\round{Z_{n}^{1+\delta_1}+Z_{n}^{1+\delta_2}}\quad \text{for all } n\in\mathbb{N}_0,
	\end{align}
	where
	\begin{align}\label{S-mu}
		\begin{split}
			0&<\mu_1:=\frac{\eps r^{-}}{r^++s^++\eps}<\mu_2:=r^+(1+\gamma_2)+\frac{\eps r^{+}}{r^++s^++\eps},\\
			1&<b:=2^{(r^++s^+)(1+\gamma_2)+\frac{\eps r^+}{r^++s^++\eps}}
		\end{split}
	\end{align}
	and
	\begin{align}\label{S-delta}
		0<\delta_1:=\gamma_1+\frac{\eps}{r^++s^++\eps}\leq \delta_2:=\gamma_2+\frac{\eps}{r^++s^++\eps}.
	\end{align}

	\vskip5pt
	\textbf{Step 3. A priori bounds.}

	In this step, we will obtain \eqref{D-bound} by using an argument similar as in Ho-Kim \cite[Proof of Theorem 4.2]{Ho-Kim-2019}. From  Lemma \ref{leRecur}, we get using \eqref{S-D-ReIneq} that
	\begin{align}\label{apply.le.1}
		Z_n\to 0\quad \text{as } n\to \infty
	\end{align}
	provided
	\begin{align}\label{apply.le.2}
		Z_0
		\leq
		\min\left\{(2C_{13}(\kappa_{*}^{-\mu_1}+\kappa_{*}^{-\mu_2}))^{-\frac{1}{\delta_1}}b^{-\frac{1}{\delta_1^2}},
		\round{2C_{13}\round{\kappa_{*}^{-\mu_1}+\kappa_{*}^{-\mu_2}}}^{-\frac{1}{\delta_2}}b^{-\frac{1}{\delta_1\delta_2}-\frac{\delta_2-\delta_1}{\delta_2^2}}\right\}.
	\end{align}
	In order to specify $\kappa_{*}$ satisfying \eqref{apply.le.2}, we first estimate
	\begin{align}\label{apply.le.2'new}
		Z_0= \int_{\Omega}\left[(u-\kappa_{*})_+^{p(x)}+\mu(x)(u-\kappa_{*})_+^{q(x)}\right]\,\diff x\le
		\int_{\Omega}\Psi(x,|u|)\,\diff x.
	\end{align}
	Note that
	\begin{align}\label{apply.le.3}
		\begin{split}
			\int_{\Omega}\Psi(x,|u|)\,\diff x
			&\leq (2C_{13})^{-\frac{1}{\delta_1}}(\kappa_{*}^{-\mu_1}+\kappa_{*}^{-\mu_2})^{-\frac{1}{\delta_1}}b^{-\frac{1}{\delta_1^2}},\\
			\int_{\Omega}\Psi(x,|u|)\,\diff x
			&\leq
			(2C_{13})^{-\frac{1}{\delta_2}}(\kappa_{*}^{-\mu_1}+\kappa_{*}^{-\mu_2})^{-\frac{1}{\delta_2}}b^{-\frac{1}{\delta_1\delta_2}-\frac{\delta_2-\delta_1}{\delta_2^2}}
		\end{split}
	\end{align}
	is equivalent to
	\begin{align*}
		\kappa_{*}^{-\mu_1}+\kappa_{*}^{-\mu_2}
		&\leq (2C_{13})^{-1}b^{-\frac{1}{\delta_1}}\round{\int_{\Omega}\Psi(x,|u|)\,\diff x}^{-\delta_1},\\
		\kappa_{*}^{-\mu_1}+\kappa_{*}^{-\mu_2}
		&\leq
		(2C_{13})^{-1}b^{-\frac{1}{\delta_1}-\frac{\delta_2-\delta_1}{\delta_2}}\round{\int_{\Omega}\Psi(x,|u|)\,\diff x}^{-\delta_2}.
	\end{align*}
	On the other hand we have that
	\begin{align*}
		2\kappa_{*}^{-\mu_1}
		&\leq (2C_{13})^{-1}b^{-\frac{1}{\delta_1}-\frac{\delta_2-\delta_1}{\delta_2}}\min\left\{\round{\int_{\Omega}\Psi(x,|u|)\,\diff x}^{-\delta_1}, \round{\int_{\Omega}\Psi(x,|u|)\,\diff x}^{-\delta_2}\right\},\\
		2\kappa_{*}^{-\mu_2}
		&\leq (2C_{13})^{-1}b^{-\frac{1}{\delta_1}-\frac{\delta_2-\delta_1}{\delta_2}}\min\left\{\round{\int_{\Omega}\Psi(x,|u|)\,\diff x}^{-\delta_1},
		\round{\int_{\Omega}\Psi(x,|u|)\,\diff x}^{-\delta_2}\right\},
	\end{align*}
	is equivalent to
	\begin{align}\label{apply.le.4}
		\begin{split}
			\kappa_{*}
			&\geq (4C_{13})^{\frac{1}{\mu_1}}b^{\frac{1}{\mu_1}\round{\frac{1}{\delta_1}+\frac{\delta_2-\delta_1}{\delta_2}}}\max\left\{\round{\int_{\Omega}\Psi(x,|u|)\,\diff x}^{\frac{\delta_1}{\mu_1}}, \round{\int_{\Omega}\Psi(x,|u|)\,\diff x}^{\frac{\delta_2}{\mu_1}}\right\},\\
			\kappa_{*}
			&\geq
			(4C_{13})^{\frac{1}{\mu_2}}b^{\frac{1}{\mu_2}\round{\frac{1}{\delta_1}+\frac{\delta_2-\delta_1}{\delta_2}}}\max\left\{\round{\int_{\Omega}\Psi(x,|u|)\,\diff x}^{\frac{\delta_1}{\mu_2}},
			\round{\int_{\Omega}\Psi(x,|u|)\,\diff x}^{\frac{\delta_2}{\mu_2}}\right\}.
		\end{split}
	\end{align}
	Therefore, by choosing
	\begin{align*}
		\kappa_{*}=\max\left\{(4C_{13})^{\frac{1}{\mu_1}},(4C_{13})^{\frac{1}{\mu_2}}\right\}b^{\frac{1}{\mu_1}\round{\frac{1}{\delta_1}+\frac{\delta_2-\delta_1}{\delta_2}}}\max\left\{\round{\int_{\Omega}\Psi(x,|u|)\,\diff x}^{\frac{\delta_1}{\mu_2}},
		\round{\int_{\Omega}\Psi(x,|u|)\,\diff x}^{\frac{\delta_2}{\mu_1}}\right\},
	\end{align*}
	\eqref{apply.le.4} holds and so \eqref{apply.le.3} follows. From this and \eqref{apply.le.2'new}, we derive \eqref{apply.le.2}. Hence, \eqref{apply.le.1} holds. Meanwhile, by Lebesgue's dominated convergence theorem, we have
	\begin{align*}
		Z_n&=\int_{\Omega}\left[(u-\kappa_{n})_+^{r(x)}+\mu(x)^{\frac{s(x)}{q(x)}}(u-\kappa_{n})_+^{s(x)}\right]\,\diff x\\
		&\to \int_{\Omega}\left[(u-2\kappa_{*})_+^{r(x)}+\mu(x)^{\frac{s(x)}{q(x)}}(u-2\kappa_{*})_+^{s(x)}\right]\,\diff x
		\quad \text{as }  n\to \infty.
	\end{align*}
	This implies that
	\begin{align*}
		\int_{\Omega}\left[(u-2\kappa_{*})_+^{r(x)}+\mu(x)^{\frac{s(x)}{q(x)}}(u-2\kappa_{*})_+^{s(x)}\right]\,\diff x=0
	\end{align*}
	and so
	\begin{align*}
		\esssup_{x\in\Omega} u(x)\leq 2\kappa_{*}.
	\end{align*}

	Replacing $u$ with $-u$ in the arguments above, we obtain
	\begin{align*}
		\esssup_{x\in\Omega} (-u)(x)\leq 2\kappa_{*}.
	\end{align*}
	Therefore,
	\begin{align}\label{applying.le.5}
		\|u\|_{\infty,\Omega}\le
		C\max\left\{\round{\int_{\Omega}\Psi(x,|u|)\,\diff x}^{\frac{\delta_1}{\mu_2}},
		\round{\int_{\Omega}\Psi(x,|u|)\,\diff x}^{\frac{\delta_2}{\mu_1}}\right\},
	\end{align}
	where $C$ is a positive constant independent of $u$. Note that by Proposition \ref{prop_mod-nor}, we have the following relation
	\begin{align*}
		\int_{\Omega}\Psi(x,|u|)\,\diff x\leq \max\left\{\|u\|_{\Psi,\Omega}^{r^-},\|u\|_{\Psi,\Omega}^{s^+}\right\}.
	\end{align*}
	Combining this and \eqref{applying.le.5}, we derive \eqref{D-bound} and the proof is complete.
\end{proof}

Next, we want to prove a priori bounds for problem \eqref{N}. We have the following result.

\begin{theorem}\label{N.a-priori}
	Let hypotheses \eqref{H3}, \eqref{D1} and \eqref{N1} be satisfied. Then, any weak solution $u\in \Wp{\mathcal{H}}$ of problem \eqref{N} belongs to $L^\infty(\Omega)\cap L^\infty(\Gamma)$ and satisfies the following a priori estimate
	\begin{align}\label{N.L^infity}
		\|u\|_{\infty,\Omega}+\|u\|_{\infty,\Gamma} \leq C \max \left\{\left(\|u\|_{\Psi,\Omega}+\|u\|_{\Upsilon,\Gamma}\right)^{\tau_1},\left(\|u\|_{\Psi,\Omega}+\|u\|_{\Upsilon,\Gamma}\right)^{\tau_2} \right\},
	\end{align}
	where $C,\tau_1,\tau_2$ are positive constants independent of $u$ and
	\begin{align*}
		\Psi(x,t)&:=t^{r(x)}+\mu(x)^{\frac{s(x)}{q(x)}}t^{s(x)},\quad
		\Upsilon(x,t):=t^{\ell(x)}+\mu(x)^{\frac{h(x)}{q(x)}}t^{h(x)}
	\end{align*}
	for all $(x,t)\in\overline{\Omega}\times [0,\infty)$.
\end{theorem}

\begin{proof}
	The proof uses similar ideas as the proof of Theorem \ref{D.a-priori}.

	\vskip5pt
	\textbf{Step 1. Defining the recursion sequence $\{X_n\}_{n\in\N_0}$ and basic estimates.}

	For each $n\in\N_0$, we define
	\begin{align*}
		X_n:=Z_n+Y_n,
	\end{align*}
	where
	\begin{align*}
		Z_n:=\int_{A_{\kappa_{n}} }\left[(u-\kappa_{n})^{r(x)}+\mu(x)^{\frac{s(x)}{q(x)}}(u-\kappa_{n})^{s(x)}\right]\,\diff x
	\end{align*}
	and
	\begin{align*}
		Y_n:=\int_{\Gamma_{\kappa_{n}} }\left[(u-\kappa_{n})^{\ell(x)}+\mu(x)^{\frac{h(x)}{q(x)}}(u-\kappa_{n})^{h(x)}\right]\,\diff \sigma.
	\end{align*}
	Here $A_\kappa$ and $\{\kappa_{n}\}_{n\in\N_0}$ are given by \eqref{def.Ak} and \eqref{def.kn}, respectively, and
	\begin{align}\label{def.Gammak}
		\Gamma_\kappa:=\{x\in\Gamma\,:\, u(x)>\kappa\},\quad \kappa\in\R.
	\end{align}
	It is also clear that
	\begin{align*}
		\Gamma_{\kappa_{n+1}}\subset \Gamma_{\kappa_{n}}
		\quad\text{and}\quad
		Y_{n+1}\leq Y_n\quad \text{for all } n\in \N_0.
	\end{align*}
	Arguing as that obtained in \eqref{S-est.u} and \eqref{S-|A_{k_{n+1}}|} we have
	\begin{align} \label{N.S-est.u}
		u(x)\leq (2^{n+2}-1)(u(x)-\kappa_{n})\quad \text{for a.\,a.\,}x\in \Gamma_{\kappa_{n+1}}\text{ and for all } n\in\N_0
	\end{align}
	and
	\begin{align} \label{S-|Ga_{k_{n+1}}|}
		|\Gamma_{\kappa_{n+1}}|_\sigma \leq \left(\kappa_{*}^{-\ell^{-}}+\kappa_{*}^{-\ell^{+}}\right)2^{(n+1)\ell^+} Y_n\leq 2\left(1+\kappa_{*}^{-\ell^{+}}\right)2^{(n+1)\ell^+} Y_n
		\quad \text{for all } n\in\N_0.
	\end{align}
	Furthermore, we are going to show the following truncated energy estimate
	\begin{align}\label{N.S-grad}
		\int_{A_{\kappa_{n+1}}}\left[|\nabla u|^{p(x)}+\mu(x)|\nabla u|^{q(x)}\right]\,\diff x\leq  C_1(1+\kappa_{*}^{-\alpha_0})2^{n\beta_0}X_n\quad \text{for all } n\in\N_0,
	\end{align}
	where $\alpha_0:=\max\{r^+,\ell^+\}$ and $\beta_0:=\max\{r^+,s^+,\ell^+,h^+\}$. As before, we denote by $C_i$ ($i\in\N$) positive constants independent of $u$, $n$ and $\kappa_{*}$. In order to prove \eqref{N.S-grad}, we test \eqref{def_sol_N} with $\varphi=(u-\kappa_{n+1})_{+} \in W^{1,\mathcal{H}}(\Omega)$ in order to get
	\begin{align}\label{N.S-D.var.Eq}
		\begin{split}
			&\int_{A_{\kappa_{n+1}}}
			\mathcal{A}(x,u,\nabla u) \cdot \nabla u\,\diff x\\
			&=\int_{A_{\kappa_{n+1}} }\mathcal{B}(x,u,\nabla u)(u-\kappa_{n+1})\,\diff x
			+\int_{\Gamma_{\kappa_{n+1}}}\mathcal{C}(x,u)(u-\kappa_{n+1})\,\diff\sigma.
		\end{split}
	\end{align}
	From the previous subsection, by using the structure conditions in \ref{D1ii} and \ref{D1iii}, we have
	\begin{align*}
		&\int_{A_{\kappa_{n+1}}} \mathcal{A}(x,u,\nabla u)\cdot \nabla u\,\diff x \\
		&\geq \alpha_2\int_{A_{\kappa_{n+1}} }\left[|\nabla u|^{p(x)}+\mu(x)|\nabla u|^{q(x)}\right]\,\diff x-\alpha_3\int_{A_{\kappa_{n+1}} }\left[u^{r(x)}+\mu(x)^{\frac{s(x)}{q(x)}}u^{s(x)}+1\right]\,\diff x
	\end{align*}
	and
	\begin{align*}
		&\int_{A_{\kappa_{n+1}}} \mathcal{B}(x,u,\nabla u)(u-\kappa_{n+1})\,\diff x\\
		&\leq \frac{\alpha_2}{2}\int_{A_{\kappa_{n+1}} }\left[|\nabla u|^{p(x)}+\mu(x)|\nabla u|^{q(x)}\right]\,\diff x+C_2\int_{A_{\kappa_{n+1}} }\left[u^{r(x)}+\mu(x)^{\frac{s(x)}{q(x)}}u^{s(x)}+1\right]\,\diff x.
	\end{align*}
From the fact that $0<u-\kappa_{n+1}<u\leq u^{\ell(x)}+1$ on $\Gamma_{\kappa_{n+1}}$ and hypothesis \eqref{N1} we  get
	\begin{align*}
		\int_{\Gamma_{\kappa_{n+1}} }\mathcal{C}(x,u)(u-\kappa_{n+1})\,\diff\sigma &\leq \gamma \int_{\Gamma_{\kappa_{n+1}} }\left[u^{\ell(x)-1}+\mu(x)^{\frac{h(x)}{q(x)}}u^{h(x)-1}+1\right]u\,\diff\sigma\\
		&\leq 2\gamma \int_{\Gamma_{\kappa_{n+1}} }\left[u^{\ell(x)}+\mu(x)^{\frac{h(x)}{q(x)}}u^{h(x)}+1\right]\,\diff\sigma.
	\end{align*}
	Combining the last three estimates with \eqref{N.S-D.var.Eq} and then using \eqref{S-est.u} and \eqref{N.S-est.u}, we obtain
	\begin{align*}
		&\int_{A_{\kappa_{n+1}} }\left[|\nabla u|^{p(x)}+\mu(x)|\nabla u|^{q(x)}\right]\,\diff x\\
		&\leq C_3\int_{A_{\kappa_{n+1}} }\left[u^{r(x)}+\mu(x)^{\frac{s(x)}{q(x)}}u^{s(x)}+1\right]\,\diff x+C_4\int_{\Gamma_{\kappa_{n+1}} }\left[u^{\ell(x)}+\mu(x)^{\frac{h(x)}{q(x)}}u^{h(x)}+1\right]\,\diff\sigma\\
		&\leq C_3\int_{A_{\kappa_{n+1}}}\left(\left[(2^{n+2}-1)(u-\kappa_{n})\right]^{r(x)}+\mu(x)^{\frac{s(x)}{q(x)}}\left[(2^{n+2}-1)(u-\kappa_{n})\right]^{s(x)}\right)\,\diff x+C_3|A_{\kappa_{n+1}}|\\
		& \quad +C_4\int_{\Gamma_{\kappa_{n+1}}}\left(\left[(2^{n+2}-1)(u-\kappa_{n})\right]^{\ell(x)}+\mu(x)^{\frac{h(x)}{q(x)}}\left[(2^{n+2}-1)(u-\kappa_{n})\right]^{h(x)}\right)\,\diff \sigma+C_4|\Gamma_{\kappa_{n+1}}|_\sigma.
	\end{align*}
	This yields
	\begin{align*}
		&\int_{A_{\kappa_{n+1}} }\left[|\nabla u|^{p(x)}+\mu(x)|\nabla u|^{q(x)}\right]\,\diff x\leq C_52^{n\beta_0}X_n+C_3|A_{\kappa_{n+1}}|+C_4|\Gamma_{\kappa_{n+1}}|_\sigma.
	\end{align*}
	Combining this with \eqref{S-|A_{k_{n+1}}|} and \eqref{S-|Ga_{k_{n+1}}|} we obtain \eqref{N.S-grad}.

	\vskip5pt
	\textbf{Step 2. Estimating $X_{n+1}$ by $X_n$.}

	Let $n\in\N_0$. We are going to estimate $X_{n+1}$ by $X_n$ through estimating $Z_{n+1}$ by $Z_n$ and $Y_{n+1}$ by $X_n$. To this end, let $\{B_i\}_{i=1}^m$ be a finite open cover of $\overline{\Omega}$ as in Step 2 of the proof of Theorem \ref{D.a-priori} with the same notations. Denote by $I$ the set of all $i\in\{1,\cdots,m\}$ such that $\Gamma_i:=B_i\cap \Gamma\ne\emptyset$. We may take $R$ such that \eqref{S-D.Omega_i} and \eqref{S-D.loc.exp} hold and
	\begin{align}\label{S-N.Omega_i}
		|\partial\Omega_{i}|_\sigma<1\quad \text{for all } i\in I
	\end{align}
	and
	\begin{align}\label{S-N.loc.exp}
		p_i^+<\ell_i^-\leq\ell_i^+<(p_*)_i^-\ \ \text{and} \ \ q_i^+<h_i^-\leq h_i^+<(q_*)_i^-\quad \text{for all } i\in I.
	\end{align}
	From Step 2 of the proof of Theorem \ref{D.a-priori}, we have 	\begin{align}\label{S-N-ReIneqZ}
		Z_{n+1}\le
		C_{6}\round{\kappa_{*}^{-\mu_1}+\kappa_{*}^{-\mu_2}}b^n\round{Z_{n}^{1+\delta_1}+Z_{n}^{1+\delta_2}}\quad \text{for all } n\in\mathbb{N}_0,
	\end{align}
	where $b,\mu_1,\mu_2$ are given by \eqref{S-mu} and $\delta_1,\delta_2$ are given by \eqref{S-delta}.

	Next we estimate $Y_{n+1}$ by $X_n$. For each $i\in I$, $\hat{\alpha}>0$, and $\hat{\beta}>0$, we denote
	\begin{align*}
		H_{n,i}(\hat{\alpha},\hat{\beta}):=\int_{\partial\Omega_{i}}\left[v_n^{\hat{\alpha}}+\mu(x)^{\frac{\hat{\beta}}{q(x)}}v_n^{\hat{\beta}}\right]\,\diff \sigma,
	\end{align*}
	where $v_n:=(u-\kappa_{n+1})_+$.
	We have
	\begin{align*}
		Y_{n+1}=\int_{\Gamma}\left[v_n^{\ell(x)}+\mu(x)^{\frac{h(x)}{q(x)}}v_n^{h(x)}\right]\,\diff \sigma\leq \sum_{i\in I}\int_{\partial\Omega_{i}}\left[v_n^{\ell(x)}+\mu(x)^{\frac{h(x)}{q(x)}}v_n^{h(x)}\right]\,\diff \sigma.
	\end{align*}
	From this and \eqref{basic.ineq} we obtain
	\begin{align}\label{S-N-decompose1}
		Y_{n+1}\leq \sum_{i\in I}\big[H_{n,i}(\ell_i^-,h_i^-)+H_{n,i}(\ell_i^+,h_i^+)\big].
	\end{align}
	From \eqref{S-N.loc.exp}, we can fix $\eps$ such that
	\begin{align}\label{S-N-eps}
		0<\eps<\min_{i\in I}\ \min\l\{(p_*)_i^--\ell_i^+,(q_*)_i^--h_i^+\r\}.
	\end{align}
	Let $i\in I$ and let $\star\in\{+,-\}$. By H\"older's inequality and \eqref{S-N.Omega_i} we have
	\begin{align}\label{S-N.de-E1}
		\begin{split}
			H_{n,i}(\ell_i^\star,h_i^\star)&=\int_{\Gamma_{\kappa_{n+1}}\cap \partial\Omega_{i}}\left[v_n^{\ell_i^\star}+\mu(x)^{\frac{h_i^\star}{q(x)}}v_n^{h_i^\star}\right]\,\diff \sigma\\
			&\leq \left(\int_{\partial\Omega_{i}}v_n^{\ell_i^\star+\eps}\,\diff \sigma\right)^{\frac{\ell_i^\star}{\ell_i^\star+\eps}}|\Gamma_{\kappa_{n+1}}\cap \partial\Omega_{i}|_\sigma^{\frac{\eps}{\ell_i^\star+\eps}}\\
			&\quad +\left(\int_{\partial\Omega_{i}}\mu(x)^{\frac{h_i^\star+\eps}{q(x)}}v_n^{h_i^\star+\eps}\,\diff \sigma\right)^{\frac{h_i^\star}{h_i^\star+\eps}}|\Gamma_{\kappa_{n+1}}\cap \partial\Omega_{i}|_\sigma^{\frac{\eps}{h_i^\star+\eps}}\\
			&\leq |\Gamma_{\kappa_{n+1}}|_\sigma^{\frac{\eps}{\ell^++h^++\eps}}\left[\left(\int_{\partial\Omega_{i}}v_n^{\ell_i^\star+\eps}\,\diff \sigma\right)^{\frac{\ell_i^\star}{\ell_i^\star+\eps}}+\left(\int_{\partial\Omega_{i}}\mu(x)^{\frac{h_i^\star+\eps}{q(x)}}v_n^{h_i^\star+\eps}\,\diff \sigma\right)^{\frac{h_i^\star}{h_i^\star+\eps}}\right].
		\end{split}
	\end{align}
We denote
	\begin{align*}
		\Phi_\star(x,t):=t^{\ell_i^\star+\eps}+\mu(x)^{\frac{h_i^\star+\eps}{q(x)}}t^{h_i^\star+\eps}
		\quad\text{for }(x,t)\in \close\times[0,\infty).
	\end{align*}
	Since $\ell_i^\star+\eps<(p_*)_i^-$ and $h_i^\star+\eps<(q_*)_i^-$ (see \eqref{S-N-eps}), we have
	\begin{align}\label{S-N.loc-emb'}
		W^{1,p(\cdot)}(\Omega_i)\hookrightarrow	W^{1,p_i^-}(\Omega_i)\hookrightarrow L^{	\ell_i^\star+\eps}(\partial\Omega_i)
	\end{align}
and
	\begin{align}\label{S-N.loc-emb}
		W^{1,\mathcal{H}}(\Omega_i)\hookrightarrow L^{\Phi_\star}(\partial\Omega_i)
	\end{align}
in view of Proposition~\ref{proposition_embeddings} and Remark~\ref{Rmk} (for the case $\Omega=\Omega_i$), and Proposition \ref{prop_S-C-T-E} (for the case $\Omega=\Omega_i$). Applying Proposition \ref{prop_mod-nor2} (for the case $\mu\equiv 0$ and $\Omega=\Omega_i$) and the embedding \eqref{S-N.loc-emb'}, we have
	\begin{align}\label{S-N.de-E1'}
		\begin{split}
			\left(\int_{\partial\Omega_{i}}v_n^{\ell_i^\star+\eps}\,\diff \sigma\right)^{\frac{\ell_i^\star}{\ell_i^\star+\eps}}\leq C_7 \|v_n\|_{1,p(\cdot),\Omega_i}^{\ell_i^\star}\leq C_8\left(R_{n}^{\frac{\ell_i^\star}{p_i^-}}+R_{n}^{\frac{\ell_i^\star}{p_i^+}}\right),
		\end{split}
	\end{align}
where
	\begin{align*}
	R_{n}:=\int_{\Omega}\left[|\nabla v_n|^{p(x)}+\mu(x)|\nabla v_n|^{q(x)}\right]\,\diff x	+\int_{\Omega}\left[ v_n^{p(x)}+\mu(x)v_n^{q(x)}\right]\,\diff x.
\end{align*}
On the other hand, by invoking the embedding \eqref{S-N.loc-emb} and Proposition \ref{prop_mod-nor2} we have
\begin{align}\label{S-N.de-E2}
	\left(\int_{\partial\Omega_{i}}\mu(x)^{\frac{h_i^\star+\eps}{q(x)}}v_n^{h_i^\star+\eps}\,\diff \sigma\right)^{\frac{h_i^\star}{h_i^\star+\eps}}
	\leq \|v_n\|_{\Phi_\star,\partial\Omega_i}^{h_i^\star}\leq C_9\|v_n\|_{1,\mathcal{H},\Omega_i}^{h_i^\star}\leq C_9\left(R_{n}^{\frac{h_i^\star}{p_i^-}}+R_{n}^{\frac{h_i^\star}{q_i^+}}\right).
\end{align}
From \eqref{S-N.de-E1}, \eqref{S-N.de-E1'} and \eqref{S-N.de-E2}, we obtain
\begin{align}\label{S-N.de-E3}
	H_{n,i}(\ell_i^\star,h_i^\star)
	\leq C_{10}|\Gamma_{\kappa_{n+1}}|_\sigma^{\frac{\eps}{\ell^++h^++\eps}}\left(R_{n}^{\frac{\ell_i^\star}{p_i^-}}+R_{n}^{\frac{\ell_i^\star}{p_i^+}}+R_{n}^{\frac{h_i^\star}{p_i^-}}+R_{n}^{\frac{h_i^\star}{q_i^+}}\right).
\end{align}
	From \eqref{S-N.de-E3} and \eqref{basic.ineq} we infer
	\begin{align}\label{S-N.de-E5}
		H_{n,i}(\ell_i^\star,h_i^\star)\leq C_{11}|\Gamma_{\kappa_{n+1}}|_\sigma^{\frac{\eps}{\ell^++h^++\eps}}\left(R_{n}^{1+\bar{\gamma}_1}+R_{n}^{1+\bar{\gamma}_2}\right),
	\end{align}
	where
	\begin{align*}
		0<\bar{\gamma}_1:=\min_{1\leq i\leq m} \min\left\{\frac{\ell_i^-}{p_i^+},\frac{h_i^-}{q_i^+}\right\}-1\leq\bar{\gamma}_2:=\max_{1\leq i\leq m} \max\left\{\frac{\ell_i^+}{p_i^-},\frac{h_i^+}{p_i^-}\right\}-1.
	\end{align*}
	Utilizing the estimate \eqref{S-N.de-E5}, we deduce from \eqref{S-N-decompose1} that
	\begin{align}\label{S-N.de-E6}
		Y_{n+1}\leq C_{11}|\Gamma_{\kappa_{n+1}}|_\sigma^{\frac{\eps}{\ell^++h^++\eps}}\left(R_{n}^{1+\bar{\gamma}_1}+R_{n}^{1+\bar{\gamma}_2}\right)
		\quad\text{for all } n\in\N_0.
	\end{align}
Note that by \eqref{S-D.u} and \eqref{N.S-grad} we have
	\begin{align*}
	R_n\leq  C_{12}(1+\kappa_{*}^{-\alpha_0})2^{n\beta_0}X_n\quad \text{for all } n\in\N_0,
\end{align*}
where $\alpha_0$ and $\beta_0$ are given in \eqref{N.S-grad}. It follows that
	\begin{align}\label{S-N.de-E7}
		R_{n}^{1+\bar{\gamma}_1}+R_{n}^{1+\bar{\gamma}_2}\leq C_{13}\left(1+\kappa_{*}^{-\alpha_0(1+\bar{\gamma}_2)}\right)2^{n\beta_0(1+\bar{\gamma}_2)}\left(X_{n}^{1+\bar{\gamma}_1}+X_{n}^{1+\bar{\gamma}_2}\right).
	\end{align}
	On the other hand, we deduce from \eqref{S-|Ga_{k_{n+1}}|} that
	\begin{align}\label{S-N.de-E8}
		|\Gamma_{\kappa_{n+1}}|_\sigma^{\frac{\eps}{\ell^++h^++\eps}}\leq C_{14}\left(\kappa_{*}^{-\frac{\eps \ell^{-}}{\ell^++h^++\eps}}+\kappa_{*}^{-\frac{\eps \ell^{+}}{\ell^++h^++\eps}}\right)2^{\frac{\eps \ell^+}{\ell^++h^++\eps}n} Y_n^{\frac{\eps}{\ell^++h^++\eps}}.
	\end{align}
	From \eqref{S-N.de-E6}, \eqref{S-N.de-E7} and \eqref{S-N.de-E8} we arrive at
	\begin{align}\label{S-N-ReIneqY}
		Y_{n+1}\le
		C_{15}\round{\kappa_{*}^{-\bar{\mu}_1}+\kappa_{*}^{-\bar{\mu}_2}}\bar{b}^n\round{X_{n}^{1+\bar{\delta}_1}+X_{n}^{1+\bar{\delta}_2}}\quad \text{for all } n\in\mathbb{N}_0,
	\end{align}
	where
	\begin{align*}
		0<\bar{\mu}_1:=\frac{\eps \ell^{-}}{\ell^++h^++\eps}&<\bar{\mu}_2:=\alpha_0(1+\bar{\gamma}_2)+\frac{\eps \ell^{+}}{\ell^++h^++\eps},\\
		1&<\bar{b}:=2^{\beta_0(1+\bar{\gamma}_2)+\frac{\eps \ell^+}{\ell^++h^++\eps}},
	\end{align*}
	and
	\begin{align*}
		0<\bar{\delta}_1:=\bar{\gamma}_1+\frac{\eps}{\ell^++h^++\eps}\leq \bar{\delta}_2:=\bar{\gamma}_2+\frac{\eps}{\ell^++h^++\eps}.
	\end{align*}
	Finally, combining \eqref{S-N-ReIneqZ} with \eqref{S-N-ReIneqY} gives
	\begin{align}\label{S-N-ReIneqX}
		X_{n+1}\le
		C_{16}\round{\kappa_{*}^{-\eta_1}+\kappa_{*}^{-\eta_2}}b_0^n\round{X_{n}^{1+\lambda_1}+X_{n}^{1+\lambda_2}}
		\quad \text{for all } n\in\mathbb{N}_0,
	\end{align}
	where
	\begin{align*}
		0<\eta_1:=\min\{\mu_1,\bar{\mu}_1\}\leq \eta_2:=\max\{\mu_2,\bar{\mu}_2\},1<b_0:=\max\{b,\bar{b}\},
	\end{align*}
	and
	\begin{align*}
		0<\lambda_1:=\min\{\delta_1,\bar{\delta}_1\}\leq \lambda_2:=\max\{\delta_2,\bar{\delta}_2\}.
	\end{align*}

	\vskip5pt
	\textbf{Step 3. A priori bounds.}

	We will also apply Lemma \ref{leRecur} to the sequence $\{X_n\}_{n\in\N_0}$ with the recursion inequality \eqref{S-N-ReIneqX}. Note that
	\begin{align*}
		\begin{split}
			X_0
			&=\int_{\Omega}\left[(u-\kappa_{*})_+^{p(x)}+\mu(x)(u-\kappa_{*})_+^{q(x)}\right]\,\diff x\\
			&\quad +\int_{\Gamma}\left[(u-\kappa_{*})_+^{\ell(x)}+\mu(x)^{\frac{h(x)}{q(x)}}(u-\kappa_{*})_+^{h(x)}\right]\,\diff \sigma\\
			&\le
			\int_{\Omega}\Psi(x,|u|)\,\diff x+\int_{\Gamma}\Upsilon(x,|u|)\,\diff \sigma
		\end{split}
	\end{align*}
	and
	\begin{align*}
		X_n
		&=\int_{\Omega}\left[(u-\kappa_{n})_+^{r(x)}+\mu(x)^{\frac{s(x)}{q(x)}}(u-\kappa_{n})_+^{s(x)}\right]\,\diff x\\
		&\quad +\int_{\Gamma}\left[(u-\kappa_{n})_+^{\ell(x)}+\mu(x)^{\frac{h(x)}{q(x)}}(u-\kappa_{n})_+^{h(x)}\right]\,\diff \sigma\\
		&\to \int_{\Omega}\left[(u-2\kappa_{*})_+^{r(x)}+\mu(x)^{\frac{s(x)}{q(x)}}(u-2\kappa_{*})^{s(x)}\right]\,\diff x\\ &\quad +\int_{\Gamma}\left[(u-2\kappa_{*})_+^{\ell(x)}+\mu(x)^{\frac{h(x)}{q(x)}}(u-2\kappa_{*})_+^{h(x)}\right]\,\diff \sigma\quad\text{as }n\to\infty.
	\end{align*}
	By repeating the arguments in Step 3 of the proof of Theorem \ref{D.a-priori} we get that
	\begin{align}\label{N.applying.le.5}
		\|u\|_{\infty,\Omega}+\|u\|_{\infty,\Gamma}\le
		C\max\left\{G(u)^{\bar{\tau}_1},G(u)^{\bar{\tau}_2}\right\},
	\end{align}
	where $C$, $\bar{\tau}_1$ and $\bar{\tau}_2$ are positive constants independent of $u$ and
	$$G(u):=\int_{\Omega}\Psi(x,|u|)\,\diff x+\int_{\Gamma}\Upsilon(x,|u|)\,\diff \sigma.$$
	In view of Proposition~\ref{prop_mod-nor} we easily derive \eqref{N.L^infity} from \eqref{N.applying.le.5} and this completes the proof.
\end{proof}

\section{The boundedness for generalized double phase problems with critical growth}\label{Section-5}

The aim of this section is to discuss the boundedness of weak solutions to the problems \eqref{D} and \eqref{N} when we suppose a critical growth based on the Propositions \ref{prop_C-do-E} and \ref{prop_C-T-E} via the idea in Ho-Kim-Winkert-Zhang \cite{Ho-Kim-Winkert-Zhang-2022} using the De Giorgi iteration together with a localization method. Let hypotheses \eqref{H3} be satisfied and we state our hypotheses on the data.

\begin{enumerate}[label=\textnormal{(D$_2$):},ref=\textnormal{D$_2$}]
	\item\label{D2}
	The functions $\mathcal{A}\colon\Omega\times\R\times\R^N\to \R^N$ and $\mathcal{B}\colon\Omega \times \R\times \R^N\to \R$ are Carath\'eodory functions such that
	\begin{enumerate}[label=\textnormal{(\roman*)},ref=\textnormal{(D$_2$)(\roman*)}]
		\item
			$ |\mathcal{A}(x,t,\xi)|
			\leq \alpha_1 \left[1+|t|^{\frac{p^*(x)}{p'(x)}}+\mu(x)^{\frac{N-1}{N-q(x)}}|t|^{\frac{q^*(x)}{q'(x)}}+|\xi|^{p(x)-1} +\mu(x)|\xi|^{q(x)-1}\right],$
		\item\label{D2ii}
			$ \mathcal{A}(x,t,\xi)\cdot \xi \geq \alpha_2 \left[|\xi|^{p(x)} +\mu(x)|\xi|^{q(x)}\right]-\alpha_3\left[|t|^{p^*(x)}+\mu(x)^{\frac{q^*(x)}{q(x)}}|t|^{q^*(x)}+1\right],$
		\item\label{D2iii}
			$ |\mathcal{B}(x,t,\xi)|
			\leq \beta \left[ |t|^{p^*(x)-1}+\mu(x)^{\frac{q^*(x)}{q(x)}}|t|^{q^*(x)-1}+|\xi|^{\frac{p(x)}{(p^*)'(x)}} +\mu(x)^{\frac{N+1}{N}}|\xi|^{\frac{q(x)}{(q^*)'(x)}}+1\right],$
	\end{enumerate}
	for a.\,a.\,$x\in\Omega$ and for all $(t,\xi) \in \R\times\R^N$, where $\alpha_1,\alpha_2,\alpha_3$ and $\beta$ are positive constants.
\end{enumerate}

For problem \eqref{N} we need an additional assumption for the boundary term.
\begin{enumerate}[label=\textnormal{(N$_2$):},ref=\textnormal{N$_2$}]
	\item\label{N2}
	The function $\mathcal{C}\colon\Gamma \times \R\to \R$ is a Carath\'eodory function such that
	\begin{enumerate}[label=\textnormal{(\roman*)},ref=\textnormal{(N$_2$)(\roman*)}]
		\item[]
		$ |\mathcal{C}(x,t)|\leq \gamma \left[|t|^{p_*(x)-1}+\mu(x)^{\frac{q_*(x)}{q(x)}}|t|^{q_*(x)-1}+1\right]$
	\end{enumerate}
	for a.\,a.\,$x\in\Gamma$ and for all $t\in \mathbb{R}$, where $\gamma$ is a positive constant.
\end{enumerate}

The definitions of weak solutions to problems \eqref{D} and \eqref{N} are the same as that given in Definition~\ref{Def.Sol}. In view of Propositions \ref{prop_C-do-E} and \ref{prop_C-T-E}, these definitions make sense under the above conditions \eqref{D2} and \eqref{N2}.

We start with the Dirichlet problem \eqref{D} and have the following result.

\begin{theorem} \label{Theo.D}
	Let hypotheses \eqref{H3} and \eqref{D2} be satisfied. Then, any weak solution $u\in \Wpzero{\mathcal{H}}$ of problem \eqref{D} is of class $L^\infty (\Omega)$.
\end{theorem}

\begin{proof}
	As before, we can cover $\overline{\Omega}$ by balls $\{B_i\}_{i=1}^m$ with radius $R$ such that each $\Omega_i:=B_i\cap\Omega$ ($i=1,\cdots,m$) is a Lipschitz domain. Note that by \eqref{H3}, it holds $q(x)<p^*(x)$ for all $x\in\overline{\Omega}$. Thus, we may take the radius $R$ sufficiently small such that
	\begin{align}\label{D.loc.exp}
		q_i^+:=\max_{x\in \overline{\Omega}_i} ~q(x)
		<{(p^*)}^{-}_{i}:=\min_{x\in \overline{\Omega}_i} p^*(x)\quad \text{for all } i\in\{1,\cdots,m\}.
	\end{align}
	Let $u$ be a weak solution to problem \eqref{D} and  let $\kappa_{*}\geq 1$ be sufficiently large such that
	\begin{align}\label{k*}
		\int_{A_{\kappa_{*}}}\mathcal{H}(x,|\nabla u|)\,\diff x+\int_{A_{\kappa_{*}}}\mathcal{H}(x,|u|)\,\diff x+\int_{A_{\kappa_{*}}}\mathcal{G}^*(x,|u|)\,\diff x<1,
	\end{align}
	where $A_\kappa$ for  $\kappa\in\R$ is defined in \eqref{def.Ak} and recall that, for all $(x,t)\in \overline{\Omega}\times [0,\infty)$,
	\begin{align*}
		\mathcal{H}(x,t):=t^{p(x)}+\mu(x)t^{q(x)}
		\quad\text{and}\quad \mathcal{G}^*(x,t):=t^{p^*(x)}+\mu(x)^{\frac{q^*(x)}{q(x)}}t^{q^*(x)},
	\end{align*}
	see \eqref{def-H} and \eqref{G^*}. Then, let $\{\kappa_{n}\}_{n\in\N_0}$ be as in \eqref{def.kn} and define $v_n:=(u-\kappa_{n+1})_+$ for each $n\in\N_0$. Moreover, we define
	\begin{align}
	\label{Zn.def}
		Z_n:=\int_{A_{\kappa_{n}}}\mathcal{H}(x,|\nabla u|)\,\diff x+\int_{A_{\kappa_{n}} }\mathcal{G}^*(x,u-\kappa_{n})\,\diff x.
	\end{align}
	Similarly to the previous Section \ref{Section-4}, we easily obtain
	\begin{align}\label{C.Zn.dec}
		Z_{n+1}\leq Z_n\quad\text{for all } n\in \mathbb{N}_0
	\end{align}
	and
	\begin{align} \label{C.|A_{k_{n+1}}|}
		|A_{\kappa_{n+1}}| \leq \frac{2^{(n+1)(p^*)^+}}{\kappa_{*}^{(p^*)^{-}}} Z_n\leq 2^{(n+1)(p^*)^+} Z_n\quad \text{for all } n\in\N_0.
	\end{align}
	In order to apply Lemma \ref{leRecur}, in the following we will establish recursion inequalities for $\{Z_n\}_{n\in\N_0}$. In the rest of the proof, as before, $C_i$ ($i\in\N$) stand for positive constants independent of $n$ and $\kappa_{*}$.

	\vskip5pt
	\textbf{Claim 1:} There exist constants $\mu_1,\mu_2>0$ such that
	\begin{align*}
		\int_{A_{\kappa_{n+1}} }\mathcal{G}^*(x,v_n)\,\diff x\leq  C_1 2^{\frac{n\left((p^*)^+\right)^2}{p^-}}\left(Z_{n}^{1+\mu_1}+Z_{n}^{1+\mu_2}\right)\quad  \text{for all } n\in\N_0.
	\end{align*}
	Indeed, we have
	\begin{align}\label{decompose1}
		\int_{A_{\kappa_{n+1}} }\mathcal{G}^*(x,v_n)\,\diff x=\int_{\Omega}\mathcal{G}^*(x,v_n)\,\diff x\leq \sum_{i=1}^m\int_{\Omega_{i}}\mathcal{G}^*(x,v_n)\,\diff x.
	\end{align}
	Let $i\in\{1,\cdots,m\}$. From \eqref{k*} and the relation between the norm and the modular (see Proposition \ref{prop_mod-nor}) we get
	\begin{align*}
		\int_{\Omega_{i}}\mathcal{G}^*(x,v_n)\,\diff x\leq \|v_n\|_{\mathcal{G}^*,\Omega_{i}}^{(p^*)_i^-}.
	\end{align*}
	Then, applying Proposition \ref{prop_C-do-E} for $\Omega=\Omega_{i}$ we obtain
	\begin{align*}
		\int_{\Omega_{i}}\mathcal{G}^*(x,v_n)\,\diff x\leq C_2\big[\|\nabla v_n\|_{\mathcal{H},\Omega_{i}}+\|v_n\|_{\mathcal{H},\Omega_{i}}\big]^{(p^*)_i^-}.
	\end{align*}
	From the equivalent norm in \eqref{equivalent-norm-W1H} and Proposition \ref{prop_mod-nor2}, noticing \eqref{k*} again, we then have
	\begin{align*}
		\int_{\Omega_{i}}\mathcal{G}^*(x,v_n)\,\diff x&\leq C_3\bigg(\int_{\Omega_{i}}\mathcal{H}(x,|\nabla v_n|)\,\diff x+\int_{\Omega_{i}}\mathcal{H}(x,v_n)\,\diff x\bigg)^{\frac{(p^*)_i^-}{q_i^+}}\\
		&\leq C_4\left(\int_{A_{\kappa_{n+1}}}\mathcal{H}(x,|\nabla v_n|)\,\diff x+\int_{A_{\kappa_{n+1}}}\mathcal{G}^*(x,v_n)\,\diff x+|A_{\kappa_{n+1}}|\right)^{\frac{(p^*)_i^-}{q_i^+}}.
	\end{align*}
	Using this along with \eqref{Zn.def}, \eqref{C.Zn.dec} and \eqref{C.|A_{k_{n+1}}|} we infer that
	\begin{align*}
		\int_{\Omega_{i}}\mathcal{G}^*(x,v_n)\,\diff x
		&\leq C_52^{\frac{n(p^*)^+(p^*)_i^-}{q_i^+}}Z_{n}^{\frac{(p^*)_i^-}{q_i^+}}.
	\end{align*}
	Now, if we combine this with \eqref{decompose1} and \eqref{basic.ineq} we arrive at
	\begin{align*}
		\int_{A_{\kappa_{n+1}}}\mathcal{G}^*(x,v_n)\,\diff x\leq C_62^{\frac{n\left((p^*)^+\right)^2}{q^-}}\left(Z_{n}^{1+\mu_1}+Z_{n}^{1+\mu_2}\right),
	\end{align*}
	where
	\begin{align*}
		0<\mu_1:=\min_{1\leq i\leq m} \frac{(p^*)_i^-}{q_i^+}-1\leq\mu_2:=\max_{1\leq i\leq m} \frac{(p^*)_i^-}{q_i^+}-1
	\end{align*}
	due to \eqref{D.loc.exp}. This shows Claim 1.

	\vskip5pt
	{\bf Claim 2:}
	It holds that
	\begin{align*}
		\int_{A_{\kappa_{n+1}}}\mathcal{H}(x,|\nabla u|)\,\diff x\leq  C_7 2^{n\l[\frac{\left((p^*)^+\right)^2}{q^-}+(q^*)^+\r]}\left(Z_{n-1}^{1+\mu_1}+Z_{n-1}^{1+\mu_2}\right)\quad \text{for all } n\in\N.
	\end{align*}

	We test  \eqref{def_sol_D} with $\varphi =v_n \in W_0^{1,\mathcal{H}}(\Omega)$ to get
	\begin{align}\label{D.var.Eq}
		\int_{A_{\kappa_{n+1}}}
		\mathcal{A}(x,u,\nabla u) \cdot \nabla u\,\diff x =\int_{A_{\kappa_{n+1}} }\mathcal{B}(x,u,\nabla u)(u-\kappa_{n+1})\,\diff x.
	\end{align}
	Note that $u\geq u-\kappa_{n+1} >0$ and $u> \kappa_{n+1} \geq 1$ on $A_{\kappa_{n+1}}$. Applying this along with the structure conditions in \ref{D2ii}, \ref{D2iii} along with Young's inequality, we reach the following estimates
	\begin{align*}
		\int_{A_{\kappa_{n+1}}} \mathcal{A}(x,u,\nabla u)\cdot \nabla u\,\diff x &\geq \alpha_2\int_{A_{\kappa_{n+1}} }\mathcal{H}(x,|\nabla u|)\,\diff x-\alpha_3\int_{A_{\kappa_{n+1}} }\left[\mathcal{G}^*(x,u)+1\right]\,\diff x\\
		&\geq \alpha_2\int_{A_{\kappa_{n+1}} }\mathcal{H}(x,|\nabla u|)\,\diff x-2\alpha_3\int_{A_{\kappa_{n+1}} }\mathcal{G}^*(x,u)\,\diff x
	\end{align*}
	and
	\begin{align*}
		&\int_{A_{\kappa_{n+1}}} \mathcal{B}(x,u,\nabla u)(u-\kappa_{n+1})\,\diff x\\
		&\leq \beta\int_{A_{\kappa_{n+1}} }\left(1+u^{p^*(x)-1}+\mu(x)^{\frac{q^*(x)}{q(x)}}u^{q^*(x)-1}+|\nabla u|^{\frac{p(x)}{(p^*)'(x)}} +\mu(x)^{\frac{N+1}{N}}|\nabla u|^{\frac{q(x)}{(q^*)'(x)}}\right)u\,\diff x\\
		&\leq \frac{\alpha_2}{2}\int_{A_{\kappa_{n+1}} }\mathcal{H}(x,|\nabla u|)\,\diff x+C_8\int_{A_{\kappa_{n+1}} }\mathcal{G}^*(x,u)\,\diff x.
	\end{align*}

	Combining the last two estimates with \eqref{D.var.Eq} and then using \eqref{S-est.u}, we obtain
	\begin{align*}
		&\int_{A_{\kappa_{n+1}} }\mathcal{H}(x,|\nabla u|)\,\diff x\\
		&\leq C_9\int_{A_{\kappa_{n+1}} }\left[u^{p^*(x)}+\mu(x)^{\frac{q^*(x)}{q(x)}}u^{q^*(x)}\right]\,\diff x\\
		&\leq C_9\int_{A_{\kappa_{n+1}}}\left(\left[(2^{n+2}-1)(u-\kappa_{n})\right]^{p^*(x)}+\mu(x)^{\frac{q^*(x)}{q(x)}}\left[(2^{n+2}-1)(u-\kappa_{n})\right]^{q^*(x)}\right)\,\diff x.
	\end{align*}
	This yields
	\begin{align*}
		\int_{A_{\kappa_{n+1}} }\mathcal{H}(x,|\nabla u|)\,\diff x
		\leq C_{10}2^{n(q^*)^+}\int_{A_{\kappa_{n}}}\mathcal{G}^*(x,v_{n-1})\,\diff x.
	\end{align*}
	Thus, from Claim 1 and the last inequality we obtain the assertion in Claim 2.

	Using the Claims 1 and 2 along with  \eqref{C.Zn.dec} gives us
	\begin{align}\label{Recur}
		Z_{n+1}\leq C_{11} b^n\left(Z_{n-1}^{1+\mu_1}+Z_{n-1}^{1+\mu_2}\right)\quad \text{for all } n\in\N,
	\end{align}
	where $b:=2^{\big[\frac{\left((p^*)^+\right)^2}{q^-}+(q^*)^+\big]}>1$. This implies
	\begin{align*}
		Z_{2(n+1)}\leq C_{11} b^{2n+1}\left(Z_{2n}^{1+\mu_1}+Z_{2n}^{1+\mu_2}\right)\quad \text{for all } n\in\N_0.
	\end{align*}
	So, writing $\widetilde{Z}_n:=Z_{2n}$ and $\widetilde{b}:=b^2$, we have
	\begin{align}\label{Recur1}
		\widetilde{Z}_{n+1}\leq bC_{11} \widetilde{b}^n\left(\widetilde{Z}_n^{1+\mu_1}+\widetilde{Z}_n^{1+\mu_2}\right) \quad \text{for all } n\in\N_0.
	\end{align}
	Now we can apply Lemma \ref{leRecur} to \eqref{Recur1}. This yields
	\begin{align}\label{Recur+1}
		Z_{2n}=\widetilde{Z}_n \to 0 \quad \text {as } n\to \infty
	\end{align}
	provided that
	\begin{align}\label{Z_0}
		\widetilde{Z}_{0}\leq \min\left\{(2bC_{11})^{-\frac{1}{\mu_1}}\ \widetilde{b}^{-\frac{1}{\mu_1^{2}}},\left(2bC_{11}\right)^{-\frac{1}{\mu_2}}\ \widetilde{b}^{-\frac{1}{\mu_1\mu_2}-\frac{\mu_2-\mu_1}{\mu_2^{2}}}\right\}.
	\end{align}
	Using again \eqref{Recur} we also obtain
	\begin{align*}
		Z_{2(n+1)+1}\leq C_{11} b^{2(n+1)}\left(Z_{2n+1}^{1+\mu_1}+Z_{2n+1}^{1+\mu_2}\right)\quad \text{for all } n\in\N_0,
	\end{align*}
	which for $\bar{Z}_n:=Z_{2n+1}$ and again $\widetilde{b}:=b^2$, reads as
	\begin{align}\label{Recur2}
		\bar{Z}_{n+1}\leq \widetilde{b}C_{11} \widetilde{b}^n\left(\bar{Z}_n^{1+\mu_1}+\widetilde{Z}_n^{1+\mu_2}\right) \quad \text{for all } n\in\N_0.
	\end{align}
	Lemma \ref{leRecur} applied to \eqref{Recur2} now leads to
	\begin{align}\label{Recur+2}
		Z_{2n+1}=\bar{Z}_n \to 0 \quad
		\text {as } n\to \infty
	\end{align}
	provided that
	\begin{align}\label{Z_0'}
		\bar{Z}_{0}\leq \min\left\{(2\widetilde{b}C_{11})^{-\frac{1}{\mu_1}}\ \widetilde{b}^{-\frac{1}{\mu_1^{2}}},\left(2\widetilde{b}C_{11}\right)^{-\frac{1}{\mu_2}}\ \widetilde{b}^{-\frac{1}{\mu_1\mu_2}-\frac{\mu_2-\mu_1}{\mu_2^{2}}}\right\}.
	\end{align}
	We point out that
	\begin{align*}
		\bar{Z}_{0}=Z_1\leq Z_0=\widetilde{Z}_0\leq \int_{A_{\kappa_{*}}}\mathcal{H}(x,|\nabla u|)\,\diff x+\int_{A_{\kappa_{*}}}\mathcal{G}^*(x,u)\,\diff x.
	\end{align*}
	Hence, if we choose $\kappa_{*}>1$ sufficiently large, we obtain
	\begin{align*}
		&\int_{A_{\kappa_{*}}}\mathcal{H}(x,|\nabla u|)\,\diff x+\int_{A_{\kappa_{*}}}\mathcal{G}^*(x,u)\,\diff x\\
		&\leq \min\left\{1,(2\widetilde{b}C_{11})^{-\frac{1}{\mu_1}}\ \widetilde{b}^{-\frac{1}{\mu_1^{2}}},\left(2\widetilde{b}C_{11}\right)^{-\frac{1}{\mu_2}}\ \widetilde{b}^{-\frac{1}{\mu_1\mu_2}-\frac{\mu_2-\mu_1}{\mu_2^{2}}}\right\}.
	\end{align*}
	Therefore, \eqref{k*}, \eqref{Z_0} and \eqref{Z_0'} are satisfied and we then get \eqref{Recur+1} and \eqref{Recur+2} which says that
	\begin{align*}
		Z_n=\int_{A_{\kappa_{n}}}\mathcal{H}(x,|\nabla u|)\,\diff x+\int_{A_{\kappa_{n}} }\mathcal{G}^*(x,u-\kappa_{n})\,\diff x\to 0\quad  \text{as } n\to\infty.
	\end{align*}
	This implies that
	\begin{align*}
		\int_{\Omega}(u-2\kappa_{*})_+^{p^*(x)}\,\diff x=0.
	\end{align*}
	Thus, $(u-2\kappa_{*})_{+}=0$ a.\,e.\,in $\Omega$ and so
	\begin{align*}
		\esssup_{\Omega} u \leq 2\kappa_*.
	\end{align*}
	Replacing $u$ by $-u$ in the above arguments we also get that
	\begin{align*}
		\esssup_{\Omega} (-u) \leq 2\kappa_*.
	\end{align*}
	From the last two estimates we obtain
	\begin{align*}
		\|u\|_{\infty,\Omega}\leq 2\kappa_{*}.
	\end{align*}
	This shows the assertion of the theorem.
\end{proof}

Next, we are going to study the boundedness of weak solutions of \eqref{N} under critical growth and the additional structure condition \eqref{N2}. This result reads as follows.

\begin{theorem} \label{Theo.N}
	Let hypotheses \eqref{H3}, \eqref{D2} and \eqref{N2} be satisfied. Then, any weak solution of problem \eqref{N} is of class $L^\infty (\Omega)\cap L^\infty (\Gamma)$.
\end{theorem}

\begin{proof}
	As in proof of Theorem \ref{Theo.D}, we cover $\overline{\Omega}$ by balls $\{B_i\}_{i=1}^m$ with radius $R$ such that each $\Omega_i:=B_i\cap\Omega$ ($i=1,\ldots,m$) is a Lipschitz domain. Denoting by $I$ the set of all $i\in\{1,\ldots,m\}$ such that $\Gamma_i:=B_i\cap \Gamma\ne\emptyset$, we may take $R$ sufficiently small such that
	\begin{align}\label{N.loc.exp2}
		p_i^+<{(p_*)}^{-}_{i}\ \ \text{and} \ \ q_i^+<{(p^*)}^{-}_{i}\quad \text{for all } i\in\{1,\cdots,m\},
	\end{align}
	where, as before, for a function $f\in C\left(\overline{\Omega}\right)$ and $i\in\{1,\cdots,m\}$ we denote
\begin{align*}
	f_i^+:=\max_{x\in\overline{\Omega}_i} f(x)
	\quad \text{and} \quad f^{-}_{i}:=\min_{x\in\overline{\Omega}_i} f(x).
\end{align*}

	Let $u$ be a weak solution to problem~\eqref{N} and let $\kappa_{*}\geq 1$ be sufficiently large such that
	\begin{align}\label{N.k*}
		\int_{A_{\kappa_{*}}}\mathcal{H}(x,|\nabla u|)\,\diff x+\int_{A_{\kappa_{*}}}\mathcal{G}^*(x,|u|)\,\diff x+\int_{\Gamma_{\kappa_{*}}}\mathcal{T}^*(x,|u|)\,\diff\sigma<1,
	\end{align}
	where $A_\kappa$ and $\Gamma_\kappa$ are defined by \eqref{def.Ak} and \eqref{def.Gammak}, respectively, and for all $(x,t)\in \overline{\Omega}\times [0,\infty)$,
	\begin{align*}
		\mathcal{G}^*(x,t):=t^{p^*(x)}+\mu(x)^{\frac{q^*(x)}{q(x)}}t^{q^*(x)}
		\quad\text{and}\quad \mathcal{T}^*(x,t):=t^{p_*(x)}+\mu(x)^{\frac{q_*(x)}{q(x)}}t^{q_*(x)}.
	\end{align*}
	Let $\{\kappa_{n}\}_{n\in\N_0}$ be as in \eqref{def.kn} and for each $n\in\N_0$, we define $v_n:=(u-\kappa_{n+1})_+$ and
	\begin{align} \label{N.Zn.def}
		Z_n:=\int_{A_{\kappa_{n}}}\mathcal{H}(x,|\nabla u|)\,\diff x+\int_{A_{\kappa_{n}} }\mathcal{G}^*(x,u-\kappa_{n})\,\diff x+\int_{\Gamma_{\kappa_n}}\mathcal{T}^*(x,u-\kappa_{n})\,\diff\sigma,
	\end{align}
	where, as before, we see that
	\begin{align}\label{N.Zn.decreasing}
		Z_{n+1}\leq Z_n\quad\text{for all } n\in \N_0,
	\end{align}
	and
	\begin{align}
	u(x)&\leq (2^{n+2}-1)(u(x)-\kappa_{n})\quad \text{for a.\,a.\,}x\in A_{\kappa_{n+1}}\quad \text{for all } n\in\N_0,\label{N.est.u1}\\
	u(x)&\leq (2^{n+2}-1)(u(x)-\kappa_{n})\quad \text{for a.\,a.\,}x\in \Gamma_{\kappa_{n+1}}\quad\text{for all } n\in\N_0,\label{N.est.u2}\\
	|A_{\kappa_{n+1}}|
	&\leq \frac{2^{(n+1)(p^*)^+}}{\kappa_{*}^{(p^*)^{-}}} Z_n\leq 2^{(n+1)(p^*)^+} Z_n\quad \text{for all } n\in\N_0.\label{N.|A_k|}
	\end{align}

	In the rest of the proof, $C_i$ ($i\in\N$) are again positive constants independent of $n$ and $\kappa_{*}$. As in the proof of Theorem~\ref{Theo.D} the following assertion holds.
	\vskip5pt
	{\bf Claim 1:}
	It holds that 	\begin{align*}
		\int_{A_{\kappa_{n+1}}}\mathcal{G}^*(x,v_n)\,\diff x\leq C_12^{\frac{n\left((p^*)^+\right)^2}{q^-}}\left(Z_{n}^{1+\nu_1}+Z_{n}^{1+\nu_2}\right)\quad \text{for all } n\in\N_0,
	\end{align*}
	where
	\begin{align*}
		0<\nu_1:=\min_{1\leq i\leq m} \frac{(p^*)_i^-}{q_i^+}-1\leq\nu_2:=\max_{1\leq i\leq m} \frac{(p^*)_i^-}{q_i^+}-1.
	\end{align*}

	We also have a similar estimate for the trace of $u$.

	\vskip5pt
	{\bf Claim 2:}
	There exist positive constants $\nu_3,\nu_4$ such that
	\begin{align*}
		\int_{\Gamma_{\kappa_{n+1}}}\mathcal{T}^*(x,v_n)\,\diff \sigma\leq  C_2 2^{\frac{n(p^*)^+(q_*)^+}{p^-}}\left(Z_{n}^{1+\nu_3}+Z_{n}^{1+\nu_4}\right)\quad \text{for all } n\in\N_0.
	\end{align*}
	Indeed, we have
	\begin{align}\label{N.deco.Bdr}
		\int_{\Gamma_{\kappa_{n+1}}}\mathcal{T}^*(x,v_n)\,\diff \sigma=\int_{\Gamma}\mathcal{T}^*(x,v_n)\,\diff \sigma\leq \sum_{i\in I}\int_{\partial\Omega_i}\mathcal{T}^*(x,v_n)\,\diff \sigma.
	\end{align}
 	Let $i\in I$. From the relation between the norm and the modular (see Propositions \ref{prop_mod-nor} and \ref{prop_mod-nor2}), the critical trace embedding for $W^{1,p(\cdot)}(\Omega_i)$ (see Proposition~\ref{proposition_embeddings} and Remark~\ref{Rmk}) and due to \eqref{N.k*}, we have
 	\begin{equation}\label{N.BE1}
 		\int_{\partial\Omega_i}v_n^{p_*(x)}\diff\sigma\leq \|v_n\|_{p_*(\cdot),\partial\Omega_i}^{(p_*)^-_i}\leq C_3\|v_n\|_{1,p(\cdot),\Omega_i}^{(p_*)^-_i}\leq C_3 R_{n}^{\frac{(p_*)^-_i}{p_i^+}},
 	\end{equation}
where
\begin{align*}
R_{n}:=\int_{\Omega}\mathcal{H}(x,|\nabla v_n|)\,\diff x+\int_{\Omega}\mathcal{H}(x,v_n)\,\diff x.
\end{align*}
Define $\varphi(x,t):=\mu(x)^{\frac{q_*(x)}{q(x)}}t^{q_*(x)}$ for $(x,t)\in\overline{\Omega}\times [0,\infty)$. From the relation between the norm and the modular, Proposition~\ref{prop_C-T-E}, and due to \eqref{N.k*} again, we have
 \begin{equation}\label{N.BE2}
 	\int_{\partial\Omega_i}\mu(x)^{\frac{q_*(x)}{q(x)}}v_n^{q_*(x)}\diff\sigma\leq \|v_n\|_{\varphi,\partial\Omega_i}^{(q_*)^-_i}\leq \|v_n\|_{\mathcal{T}^*,\partial\Omega_i}^{(q_*)^-_i}\leq C_4\|v_n\|_{1,\mathcal{H},\Omega_i}^{(q_*)^-_i}\leq C_4 R_{n}^{\frac{(q_*)^-_i}{q_i^+}}.
 \end{equation}
Combining \eqref{N.BE1} with \eqref{N.BE2} gives
	\begin{align*}
		\int_{\partial\Omega_i}\mathcal{T}^*(x,v_n)\,\diff \sigma&\leq C_3 R_{n}^{\frac{(p_*)^-_i}{p_i^+}}+C_4 R_{n}^{\frac{(q_*)^-_i}{q_i^+}}\\
		&\leq C_5\left(\int_{A_{\kappa_{n+1}}}\mathcal{H}(x,|\nabla u|)\,\diff x+\int_{A_{\kappa_{n+1}}}\mathcal{G}^*(x,v_n)\,\diff x+|A_{\kappa_{n+1}}|\right)^{\frac{(p_*)_i^-}{p_i^+}}\\
		&\quad+ C_5\left(\int_{A_{\kappa_{n+1}}}\mathcal{H}(x,|\nabla u|)\,\diff x+\int_{A_{\kappa_{n+1}}}\mathcal{G}^*(x,v_n)\,\diff x+|A_{\kappa_{n+1}}|\right)^{\frac{(q_*)_i^-}{q_i^+}}.
	\end{align*}
	From this, \eqref{N.Zn.def}, \eqref{N.Zn.decreasing} and \eqref{N.|A_k|} we obtain
	\begin{align*}
		\int_{\partial\Omega_i}\mathcal{T}^*(x,v_n)\,\diff \sigma
		\leq C_62^{\frac{n(p^*)^+(p_*)_i^-}{p_i^+}}Z_{n}^{\frac{(p_*)_i^-}{p_i^+}}+C_72^{\frac{n(p^*)^+(q_*)_i^-}{q_i^+}}Z_{n}^{\frac{(q_*)_i^-}{q_i^+}}.
	\end{align*}
	Combining this with \eqref{N.deco.Bdr} and noticing \eqref{basic.ineq} we arrive at
	\begin{align*}
		\int_{\Gamma}\mathcal{T}^*(x,v_n)\,\diff \sigma\leq C_82^{\frac{n(p^*)^+(q_*)^+}{p^-}}\left(Z_{n}^{1+\nu_3}+Z_{n}^{1+\nu_4}\right),
	\end{align*}
	where
	\begin{align*}
		0<\nu_3:=\min_{1\leq i\leq m} \min\left\{\frac{(p_*)_i^-}{p_i^+},\frac{(q_*)_i^-}{q_i^+}\right\}-1\leq\nu_4:=\max_{1\leq i\leq m} \max\left\{\frac{(p_*)_i^-}{p_i^+},\frac{(q_*)_i^-}{q_i^+}\right\}-1,
	\end{align*}
	see \eqref{N.loc.exp2}. Hence, we have proved Claim 2.

	\vskip5pt
	{\bf Claim 3:} It holds that
	\begin{align*}
		\int_{A_{\kappa_{n+1}}}\mathcal{H}(x,|\nabla u|)\,\diff x\leq  C_9 	2^{n\big[\frac{\left((p^*)^+\right)^2}{q^-}+\frac{(p^*)^+(q_*)^+}{p^-}+(q^*)^+\big]}\left(Z_{n-1}^{1+\mu_1}+Z_{n-1}^{1+\mu_2}\right)\quad \text{for all } n\in\N,
	\end{align*}
	where $0<\mu_1:=\min_{1\leq i\leq 4}\nu_i\leq \mu_2:=\max_{1\leq i\leq 4}\nu_i$.

	Testing \eqref{def_sol_N} by $\varphi =v_n \in W^{1,\mathcal{H}}(\Omega)$ gives
	\begin{align*}
		\int_{A_{\kappa_{n+1}}}
		\mathcal{A}(x,u,\nabla u) \cdot \nabla u\,\diff x =&\int_{A_{\kappa_{n+1}} }\mathcal{B}(x,u,\nabla u)(u-\kappa_{n+1})\,\diff x\\
		&+\int_{\Gamma_{\kappa_{n+1}} }\mathcal{C}(x,u)(u-\kappa_{n+1})\,\diff \sigma.
	\end{align*}
	Arguing as in the proof of Theorem~\ref{Theo.D} (see the proof of Claim 2), we obtain
	\begin{align*}
		\int_{A_{\kappa_{n+1}}} \mathcal{A}(x,u,\nabla u)\cdot \nabla u\,\diff x \geq \alpha_2\int_{A_{\kappa_{n+1}} }\mathcal{H}(x,|\nabla u|)\,\diff x-2\alpha_3\int_{A_{\kappa_{n+1}} }\mathcal{G}^*(x,u)\,\diff x,
	\end{align*}
	and
	\begin{align*}
		\int_{A_{\kappa_{n+1}}} \mathcal{B}(x,u,\nabla u)(u-\kappa_{n+1})\,\diff x\leq \frac{\alpha_2}{2}\int_{A_{\kappa_{n+1}} }\mathcal{H}(x,|\nabla u|)\,\diff x+C_{10}\int_{A_{\kappa_{n+1}} }\mathcal{G}^*(x,u)\,\diff x.
	\end{align*}
	Furthermore, by hypothesis \eqref{N2}, we have
	\begin{align*}
		\int_{\Gamma_{\kappa_{n+1}} }\mathcal{C}(x,u)(u-\kappa_{n+1})\,\diff\sigma &\leq \gamma \int_{\Gamma_{\kappa_{n+1}}} \left[u^{p_*(x)-1}+\mu(x)^{\frac{q_*(x)}{q(x)}}u^{q_*(x)-1}+1\right]u\,\diff\sigma\\
		&\leq 2\gamma \int_{\Gamma_{\kappa_{n+1}} }\mathcal{T}^*(x,u)\,\,\diff\sigma.
	\end{align*}
	Combining the last three estimates, we obtain
	\begin{align*}
		&\int_{A_{\kappa_{n+1}}} \mathcal{H}(x,|\nabla u|)\,\diff x\leq  C_{11}\int_{A_{\kappa_{n+1}} }\mathcal{G}^*(x,u)\,\diff x+C_{12}\int_{\Gamma_{\kappa_{n+1}} }\mathcal{T}^*(x,u)\,\,\diff\sigma.
	\end{align*}
	Then, by using \eqref{N.est.u1} and \eqref{N.est.u2} we deduce from the preceding inequality that
	\begin{align*}
		\int_{A_{\kappa_{n+1}}} \mathcal{H}(x,|\nabla u|)\,\diff x	\leq C_{13}2^{n(q^*)^+}\left[\int_{A_{\kappa_{n}} }\mathcal{G}^*(x,v_{n-1})\,\diff x+\int_{\Gamma_{\kappa_{n}} }\mathcal{T}^*(x,v_{n-1})\,\,\diff\sigma\right].
	\end{align*}
	Then, Claim 3 follows from the last inequality and Claims 1 and 2.

	From Claims 1--3 and \eqref{N.Zn.decreasing} we arrive at
	\begin{align}\label{N.Recur}
		Z_{n+1}\leq C_{14} b^n\left(Z_{n-1}^{1+\mu_1}+Z_{n-1}^{1+\mu_2}\right)\quad \text{for all } n\in\N,
	\end{align}
	where $b:=2^{\big[\frac{\left((p^*)^+\right)^2}{q^-}+\frac{(p^*)^+(q_*)^+}{p^-}+(q^*)^+\big]}>1$. Repeating the arguments used in the proof of Theorem \ref{Theo.D}, by choosing $\kappa_{*}>1$ sufficiently large such that
	\begin{align*}
		&\int_{A_{\kappa_{*}}}\mathcal{H}(x,|\nabla u|)\,\diff x+\int_{A_{\kappa_{*}} }\mathcal{G}^*(x,u)\,\diff x+\int_{\Gamma_{\kappa_{*}} }\mathcal{T}^*(x,u)\,\diff \sigma\\
		&\leq \min\left\{1,(2\widetilde{b}C_{14})^{-\frac{1}{\mu_1}}\ \widetilde{b}^{-\frac{1}{\mu_1^{2}}},\left(2\widetilde{b}C_{14}\right)^{-\frac{1}{\mu_2}}\ \widetilde{b}^{-\frac{1}{\mu_1\mu_2}-\frac{\mu_2-\mu_1}{\mu_2^{2}}}\right\},
	\end{align*}
	where $\widetilde{b}:=b^2$, we deduce from \eqref{N.Recur} that
	\begin{align*}
		Z_n=\int_{\Omega}\mathcal{H}(x,|\nabla v_{n-1}|)\,\diff x+\int_{\Omega}\mathcal{G}^*(x,v_{n-1})\,\diff x+\int_{\Gamma} \mathcal{T}^*(x,v_{n-1})\,\diff \sigma\to 0 \quad \text{as } n\to\infty.
	\end{align*}
	This implies that
	\begin{align*}
		\int_{\Omega}(u-2\kappa_{*})_+^{p^*(x)}\,\diff x+\int_{\Gamma}(u-2\kappa_{*})_+^{p_*(x)}\,\diff \sigma=0.
	\end{align*}
	Therefore, $(u-2\kappa_{*})_{+}=0$ a.\,e.\,in $\Omega$ and $(u-2\kappa_{*})_{+}=0$ a.\,e.\,on $\Gamma$. This means
	\begin{align*}
		 \esssup_{x\in\Omega}u(x)+\esssup_{x\in\Gamma}u(x)
		 \leq 4\kappa_{*}.
	\end{align*}

	Replacing $u$ by $-u$ in the above arguments we also obtain
	\begin{align*}
		\esssup_{x\in\Omega} (-u)(x)
		+\esssup_{x\in\Gamma} (-u)(x)
		\leq 4\kappa_{*}.
	\end{align*}

	Hence
	\begin{align*}
		\|u\|_{\infty,\Omega}+\|u\|_{\infty,\Gamma}\leq 4\kappa_{*}.
	\end{align*}
	This finishes the proof.
\end{proof}

\section*{Acknowledgment}
	The first author was supported by the University of Economics Ho Chi Minh City, Vietnam.
%
%

\end{document}